\documentclass[reqno]{amsart}
\usepackage{textcomp}
\usepackage{amsmath}
\usepackage{mathrsfs}
\usepackage{amsthm}
\usepackage[linktocpage]{hyperref}
\usepackage{amsfonts,graphics,amsthm,amsfonts,amscd,latexsym}
\usepackage{epsfig}
\usepackage{tikz}
\usepackage{tikz-cd}
\usepackage{wasysym}
\usepackage{flafter}
\usepackage{mathtools}
\usepackage{comment}
\usepackage{stmaryrd}
\usepackage{enumitem}
\usepackage{epigraph}
\usepackage[style=alphabetic]{biblatex}
\hypersetup{
colorlinks=true,
linkcolor=red,
citecolor=green,
filecolor=blue,
urlcolor=blue
}
\usepackage{tikz}
\usetikzlibrary{graphs,positioning,arrows,shapes.misc,decorations.pathmorphing}

\tikzset{
>=stealth,
every picture/.style={thick},
graphs/every graph/.style={empty nodes},
}

\tikzstyle{vertex}=[
draw,
circle,
fill=black,
inner sep=1pt,
minimum width=5pt,
]
\usepackage[position=top]{subfig}
\usepackage{amssymb}
\usepackage{color}

\calclayout

\usetikzlibrary{decorations.pathmorphing}
\tikzstyle{printersafe}=[decoration={snake,amplitude=0pt}]

\newcommand{\id}{\operatorname{id}}

\newcommand{\rk}{\operatorname{rk}}

\newcommand{\cl}{\operatorname{cl}}

\newcommand{\zz}{\mathbb{Z}}

\newcommand{\rr}{\mathbb{R}}
\newcommand{\cc}{\mathbb{C}}

\newcommand{\vgeq}{\rotatebox[origin=c]{90}{$\geqslant$}}
\newcommand{\vmapsto}{\rotatebox[origin=c]{270}{$\mapsto$}}

\newcommand{\defeq}{\vcentcolon=}

\def\O#1.{\mathcal {O}_{#1}}
\def\pr #1.{\mathbb P^{#1}}
\def\af #1.{\mathbb A^{#1}}
\def\ses#1.#2.#3.{0\to #1\to #2\to #3 \to 0}
\def\xrar#1.{\xrightarrow{#1}}
\def\K#1.{K_{#1}}
\def\bA#1.{\mathbf{A}_{#1}}
\def\bM#1.{\mathbf{M}_{#1}}
\def\bL#1.{\mathbf{L}_{#1}}
\def\bB#1.{\mathbf{B}_{#1}}
\def\bK#1.{\mathbf{K}_{#1}}
\def\subs#1.{_{#1}}
\def\sups#1.{^{#1}}

\usepackage{tikz}
\usetikzlibrary{matrix,arrows,decorations.pathmorphing}

\newtheorem{introthm}{Theorem}

\newtheorem{theorem}{Theorem}[section]
\newtheorem{lemma}[theorem]{Lemma}

\newtheorem{corollary}[theorem]{Corollary}

\theoremstyle{definition}
\newtheorem{definition}[theorem]{Definition}
\newtheorem{example}[theorem]{Example}

\newtheorem{question}[theorem]{Question}

\newtheorem{remark}[theorem]{Remark}

\renewcommand{\phi}{\varphi}

\theoremstyle{remark}

\numberwithin{equation}{section}

\usepackage[all]{xy}

\newcounter{rownumber}[figure]
\newcounter{rownumber-irr}[figure]
\newcounter{rownumber-p1}[figure]
\begin{document}

\title[Amalgams, fibre products and correspondences]{Amalgams of matroids, fibre products \\ and tropical graph correspondences}
\begin{abstract}
We prove that the proper amalgam of matroids $M_1$ and $M_2$ along their common restriction $N$ exists if and only if the tropical fibre product of Bergman fans ${B(M_1) \times_{B(N)} B(M_2)}$ is positive. We introduce tropical correspondences between Bergman fans as tropical subcycles in their product, similar to correspondences in algebraic geometry, and  define a ``graph correspondence'' of the map of lattices. We prove that graph construction is a functor for the ``covering'' maps of lattices, exploiting a generalization of Bergman fan which we call a ``Flag fan''.
\end{abstract}

\author[D.~Mineev]{Dmitry Mineev}
\address{Bar-Ilan University, Ramat Gan, Israel}
\email{dmitry.mineyev@gmail.com}

\maketitle

\section{Introduction}
\par One of the features of matroid theory is that it makes use of numerous inductive constructions such as deletions, contractions and extensions (see, for example, \cite{Oxley}, Chapter 7). In this light, various attempts to ``glue'' matroids together were studied for a long time. The most basic construction having its roots in graphic matroids is the parallel connection, which sometimes plays even bigger role than the direct sum (\cite{KrisWerner}). Generalizations of parallel connection, called amalgams, were studied in the 1970s and 1980s (\cite{BachemKern}, \cite{PoljakTurzik}). Among those, the free amalgam stands out as the one satisfying a certain universal property in the category of matroids with morphisms being what is called weak maps. A slightly stronger notion of the free amalgam, the proper amalgam, is naturally called ``generalized parallel connection'' in some cases.
\par It turns out, though, that in general amalgamation is complicated --- for a given pair of matroids that we try to glue, there may be no free amalgam, or the free amalgam may not be proper, or no amalgams may exist at all. Several criteria guaranteeing the existence of the proper amalgam were established (\cite{Oxley}, Section 11.4), but even the weakest of them is not a necessary condition. Amalgams continue to be investigated --- for example, the so-called sticky matroid conjecture was apparently resolved recently (\cite{Bonin}, \cite{Shin}).
\par Over the last thirty years, tropical geometry, initiated as a combinatorial tool of algebraic geometry inspired by the formalism of toric geometry, has established many connections to matroid theory. From the tropical point of view, a Bergman fan --- the tropical variety corresponding to a loopless matroid --- can be considered an analog of the smooth tangent cone. Bergman fans are widely used in many of the latest developments (\cite{Huh}, \cite{CSM}, \cite{BEST}).
\par In particular, in the beginning of the 2010s, a well-behaved intersection theory of the tropical cycles in Bergman fans was developed simultaneously by Allermann, Esterov, Francois, Rau, and Shaw (\cite{Esterov}, \cite{Kris}, \cite{ARau}, \cite{FRau}). This intersection theory possesses many similarities with the intersection theory in algebraic geometry, going as far as the projection formula. This machinery was sufficient to define a tropical subcycle in the product of Bergman fans called the tropical fibre product (\cite{FHampe}, \cite{TFP2}), which the authors used to study moduli spaces of curves. They considered the case when the support of the tropical fibre product equals the set-theoretic fibre product of the supports of the factors, finding several sufficient conditions on matroids for this case to occur.
\par As it turns out, though, the tropical fibre product is worth investigating in wider generality. This becomes evident when considering a simple Example \ref{ex:tfps} (2).
\begin{introthm}[\ref{th:tfp}]
The tropical fibre product equals the Bergman fan of the proper amalgam if the proper amalgam exists, and has cones with negative weights otherwise.
\end{introthm}
Thus, unlike all criteria for the existence of the proper amalgam known to us, Theorem \ref{th:tfp} is an equivalence. It is also constructive and, moreover, algorithmic in a certain sense, as shown in Section \ref{section:flagfans}.
\par With tropical counterpart of the proper amalgam in our hands, we seek a categorical justification of its name --- the fibre product. As it is justly pointed out in \cite{FHampe}, the tropical fibre product is not a categorical pullback even in the underlying category of sets. Thus, we aim to extend the notion of morphism of tropical fans to something that does not have to be a map between their supports. This leads naturally to the definition of the tropical correspondences category (Definition \ref{def:tropcor}). It formalizes the intuition behind ``not exactly graphs'' inspired by Example \ref{ex:simpletrcor}.
\par The whole category of tropical correspondences between Bergman fans is too unwieldy. In particular, since we allow the cones of the cycles to have negative weights, we cannot even require correspondences with negative weights to be irreducible. Thus, in Sections \ref{section:flagfans} and \ref{sect:tropicalgraphs} we develop a subcategory of ``covering graph correspondences'' (Definition \ref{def:graphcor}, Lemma \ref{lemma:covcorisff}) which is convenient from the combinatorial point of view. These covering graph correspondences between Bergman fans contain strictly more information than weak maps between the groundsets (Remark \ref{remark:notpullback}).
\begin{introthm}[\ref{th:functor}]
$$\{\mbox{weak lattice map}\} \to \{\mbox{graph correspondence}\}$$
is a functor from the category of covering lattice maps of flats of simple matroids with the usual composition to the category of tropical correspondences between Bergman fans.
\end{introthm}
\par The paper is organized as follows. In Section \ref{sect:prelim}, we recall necessary notions from matroid theory and tropical geometry, including the tropical fibre product from \cite{FHampe} --- Definition \ref{def:tfp}. In Section \ref{section:flagfans}, we introduce the generalization of Bergman fan called a Flag fan and observe its nice behavior with respect to Weil divisors of rational functions generalizing characteristic functions of modular cuts (Lemma \ref{lemma:ff}). Section \ref{sect:tfp} is solely devoted to the proof of Theorem \ref{th:tfp}. In Section \ref{sect:tropicalgraphs}, we introduce and study tropical correspondences, construct covering graph correspondences, establish their useful combinatorial description (Theorem \ref{th:grcorstr}), and verify the functoriality (Theorem \ref{th:functor}). We conclude with an open Question \ref{question}. Resolving it positively will achieve the best scenario still plausible --- that there is a full forgetful functor from the reasonable subcategory of tropical correspondences onto the category of weak maps.
\par I am grateful to Evgeniya Akhmedova, Alexander Esterov, Tali Kaufman, Anna--Maria Raukh, and Kris Shaw for useful discussions. I want to thank Alexander Zinov for helping me with software calculations which often preceded proofs.

\section{Preliminaries} \label{sect:prelim}
\subsection{Matroid theory} In this subsection we recall the necessary notions from the combinatorial matroid theory.
\begin{definition} \label{def:matroid}
A \textit{matroid} $M$ on the \textit{groundset} $E$ is a collection of subsets $\mathcal{F}(M) = \mathcal{F} \subset 2^E$ called \textit{flats} with the following properties:
\begin{itemize}
\item $E \in \mathcal{F}$;
\item $F_1,F_2 \in \mathcal{F} \Rightarrow F_1 \cap F_2 \in \mathcal{F}$;
\item $\forall F \in \mathcal{F} \colon E \setminus F = \bigsqcup_{F' \gtrdot F} F' \setminus F$, where $X \gtrdot Y$ is read ``$X$ \textit{covers} $Y$'' and means $X \supsetneq Y , \; \nexists Z \colon X \supsetneq Z \supsetneq Y$.
\end{itemize}
A matroid is called \textit{loopless} if $\varnothing \in \mathcal{F}$. It is called \textit{simple} if, additionally, all flats covering $\varnothing$ are one-element subsets (in this case all one-element subsets are flats). All matroids in this paper are assumed loopless.
A \textit{restriction} of the matroid $M$ on the groundset $E$ onto the subset $T \subset E$, denoted by $M|_T$, is an image of the map $\mathcal{F} \to 2^T$ given by $F \mapsto F \cap T$. An \textit{extension} of $M$ onto $E' \supset E$ is a matroid $M'$ on $E'$ such that $M'|_{E} = M$. A \textit{contraction} of $F \in \mathcal{F}(M)$ is a matroid denoted by $M/F$ on the groundset $E \setminus F$ with flats $F' \setminus F$ for $F \subset F' \in \mathcal{F}(M)$. 
\end{definition}
There is a large number of equivalent definitions of a matroid, usually called cryptomorphisms. Many of them can be found in \cite{Oxley}, Chapter 1. It can be shown that $\mathcal{F}$ forms a ranked poset with respect to the partial order given by set inclusion. Since intersection of flats is also a flat, there exists an operator of \textit{closure} $\cl \colon 2^E \to 2^E$, where $\cl_M(X) = \cap_{F \supset X} F$. Moreover, for each pair of flats $F_1$ and $F_2$ there exists a unique minimal element $F \in \mathcal{F}$ such that $F \supset F_1,F_2$, called \textit{join} and denoted by $F_1 \vee F_2$. The \textit{rank function} can be extended from $\mathcal{F}$ to the whole $2^E$ by defining 
$$\rk_M X \defeq \min_{F \supset X}{\rk F}.$$ A subset $X \subset E$ is called \textit{independent} if $\rk X = |X|$. The \textit{rank} of $M$ is defined to be $\rk_M E$. The \textit{hyperplanes} are flats of rank $\rk M - 1$.
\par
Throughout the paper, we sometimes describe matroids via pictures instead of listing their flats. The points correspond to the elements of the groundset, while lines and planes correspond to flats. Those are instances of matroids \textit{representable} over $\rr$ drawn in the projectivization of the vector space of representation (more on that in \cite{Oxley}, Chapter 1).
\par Next we prove a simple fact needed in Sections \ref{sect:tfp} and \ref{sect:tropicalgraphs}:
\begin{lemma} \label{lemma:rankunion}
The difference of ranks can only decrease when adding elements. More precisely, if $A \supset B$, then
$$\rk_M (A \cup C) - \rk_M (B \cup C) \leqslant \rk_M A - \rk_M B.$$
\end{lemma}
\begin{proof}
We can add elements of $C$ one by one, so take any $x \in C$ and let us prove that
$$\rk_M (A \cup \{x\}) - \rk_M (B \cup \{x\}) \leqslant \rk_M A - \rk_M B.$$
The rank function can only increase by 1 when adding an element (\cite{Oxley}, Section 1), and this happens exactly when this element does not belong to the closure of the set, so what we need follows from the implication
$$x \notin \cl A \Rightarrow x \notin \cl B,$$
which is true, since $A \supset B \Rightarrow \cl A \supset \cl B$.
\end{proof}
\begin{definition} \label{def:pramalg}
Consider sets $E_1,E_2, T = E_1 \cap E_2$ and matroids $M_1$ on $E_1$, $M_2$ on $E_2$, $N$ on $T$ such that $M_1|_T = N = M_2|_T$. A matroid $M$ on $E_1 \cup E_2$ is called an \textit{amalgam} of $M_1$ and $M_2$ along $N$ if $M|_{E_1} = M_1$ and $M|_{E_2} = M_2$.
\par An amalgam $M$ is called \textit{free} if for every amalgam $M'$ and every independent set $X$ in it, $X$ is also independent in $M$.
\par An amalgam $M$ is called \textit{proper} if, for every $F \in \mathcal{F}(M)$ the following holds:
$$\rk_M(F) = \eta(F) \defeq \rk_{M_1}(F \cap E_1) + \rk_{M_2}(F \cap E_2) - \rk_{N}(F \cap T).$$
\end{definition}
Both free and proper amalgams are unique, if they exist. This is not always the case, though: sometimes there is no free amalgam, sometimes no amalgams at all (see \cite{Oxley}, Section 11.4). It can be shown that the proper amalgam is always free (\cite{Oxley}, Propositions 11.4.2 and 11.4.3).
\begin{remark}
In \cite{Oxley} Proposition 11.4.2 defines proper amalgam as an amalgam where the rank of any set $X \subset E$ satisfies $\rk X = \min{\{\eta(Y) \colon Y \supseteq X\}}$, and Proposition 11.4.3 establishes equivalence with the definition that we give in Definition \ref{def:pramalg}.
\end{remark}
The free amalgam is not always proper, though, as shown by the following
\begin{example} \label{ex:cross}
\begin{figure}[h]
\includegraphics[width=\textwidth]{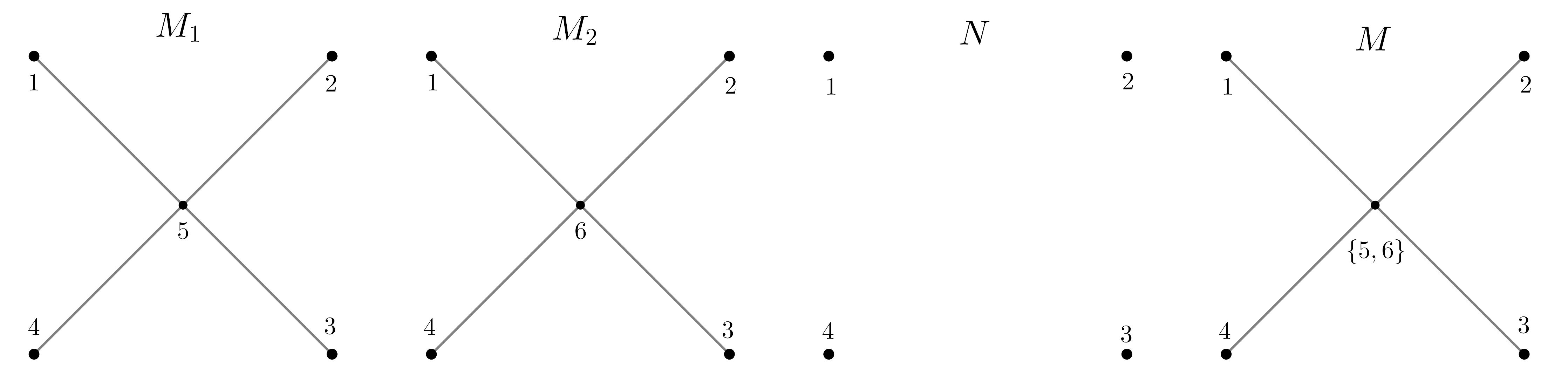}
\caption{$M$ is the free amalgam but not the proper.}
\label{fig:cross}
\end{figure}
For $M_1,M_2,N$ as shown on Figure \ref{fig:cross} $M$ is the only amalgam and hence free, but
$$1 = \rk_M \{5,6\} \neq \rk_{M_1} \{5\} + \rk_{M_2} \{6\} - \rk_N \varnothing = 2$$
\end{example}
\begin{definition} \label{def:strongweakmaps}
A map $f \colon E_1 \to E_2$ between groundsets of matroids $M_1$ and $M_2$ is called a \textit{strong map} if $f^{-1}(F) \in \mathcal{F}(M_1)$ for $F \in \mathcal{F}(M_2)$. A strong map is called an \textit{embedding} if $M_2|_{f(E_1)}=M_1$.
\par
A map $f \colon E_1 \to E_2$ between groundsets of matroids $M_1$ and $M_2$ is called a \textit{weak map} if for any $X \in E_1 \rk_{M_1}(X) \geqslant \rk_{M_2}(f(X))$. Equivalently (\cite{Oxley}, Proposition 7.3.11), $f^{-1}(X)$ is independent in $M_1$ for any independent set $X$ of $M_2$. We will sometimes write $f \colon M_1 \to M_2$ instead of $f \colon E_1 \to E_2$.
\end{definition}
It can be shown that strong maps are also weak maps (\cite{Oxley}, Corollary 7.3.12). Note that the free amalgam $M$ has the following universal property in the category of matroids with weak maps. If $\iota_{1,2} \colon N \to M_{1,2}$ and $j_{1,2} \colon M_{1,2} \to M$ are embeddings, then
\begin{equation} \label{eq:univprop}
\begin{split}
\forall M', \; \mbox{embeddings } \; j'_{1,2} \colon M_{1,2} \to M' \mbox{ such that } j'_1 \circ \iota_1 = j'_2 \circ \iota_2 \\
\exists ! \; \mbox{ weak map } \; f \colon M \to M' \mbox{ such that } f \circ j_1 = j'_1, f \circ j_2 = j'_2.
\end{split}
\end{equation}
It is not, strictly speaking, a universal property of the pushout, because allowing $j'_{1,2}$ to be any weak maps may lead to rank of $N$ dropping, and then $M'$ may become ``freer'' than any amalgam.
\begin{definition}
A pair of flats $F_1,F_2 \in \mathcal{F}(M)$ is called a \textit{modular pair} if
$$\rk{F_1} + \rk{F_2} = \rk{(F_1 \vee F_2)} + \rk{(F_1 \wedge F_2)}.$$
The flat $F \in \mathcal{F}(M)$ is called \textit{modular} if it forms a modular pair with any $F' \in \mathcal{F}(M)$.
\end{definition}
\begin{definition}
A subset $\mathcal{M}$ of flats of $M$ on groundset $E$ is called a \textit{modular cut} if it satisfies the following properties:
\begin{itemize}
\item $E \in \mathcal{M}$;
\item $F_1 \in \mathcal{M}, F_2 \ni F_1 \Rightarrow F_2 \in \mathcal{M}$;
\item if $F_1,F_2 \in \mathcal{M}$ are a modular pair, then $F_1 \wedge F_2 \in \mathcal{M}$.
\end{itemize}
\end{definition}
\subsection{Tropical geometry} In this subsection we recall the tropical counterpart of matroid theory involving the techniques needed for our results.
\begin{definition}
A (reduced or projective) \textit{Bergman fan} of matroid $M$ on groundset $E$ is a collection of maximal cones $\sigma_{\mathbf{F}}$ and their faces in $\rr^E / \langle (1, \ldots, 1) \rangle$ for each \textit{flag of flats} $\mathbf{F}$ consisting of $\varnothing = F_0 \lessdot  F_1 \lessdot  \ldots \lessdot  F_r = E$ with $F_i \in \mathcal{F}(M)$ defined as
$$\sigma_{\mathbf{F}} = \langle e_{F_1}, \ldots, e_{F_{r-1}} \rangle,$$
where $e_X$ stands for the image in $\rr^E / \langle (1, \ldots, 1) \rangle$ of the characteristic function of $X \subset E$ in $\rr^E$, i.e., $e_X(x)=1$ if $x \in X$ and $e_X(x)=0$ otherwise.
\par
A \textit{non-reduced} or \textit{affine} Bergman fan, denoted by $B(M)$, is the preimage of the projective Bergman fan in $\rr^E$.
\end{definition}
An affine Bergman fan is an instance of the main object we work with in this paper.
\begin{definition}[\cite{ARau}, Definition 2.6] \label{def:tropicalfan}
A $k$-dimensional \textit{tropical fan} $X$ in $\rr^n = \zz^n \otimes_{\zz} \rr$ is a collection of cones $\sigma$ of dimension $k$, generated by the vectors from the lattice $\zz^n$, with integer weights $\omega(\sigma)$ such that the \textit{balancing condition} holds for each $(k-1)$-dimensional face $\tau$:
$$\sum_{\sigma \supset \tau} \omega(\sigma) \cdot u_{\sigma / \tau} = 0 \in V / V_{\tau},$$
where $V_{\rho}$ is the smallest subspace of $\rr^n$ containing cone $\rho$, and $u_{\rho / \rho'}$ is a normal vector --- the generator of the one-dimensional $\zz$-module $(V_{\rho} \cap \zz^n) / (V_{\rho'} \cap \zz^n)$ chosen such that it ``looks inside'' $\rho$. The \textit{support} of the fan $X$, denoted by $|X|$, is the set of points in $\rr^n$ that belong to some cone.
\end{definition}
\par In all our considerations all the cones are going to be simplicial. Moreover, the only generator of cone $\rho$ missing from the generators of $\rho'$ is always going to map to $u_{\rho / \rho'}$ when taking quotient by $V_{\rho'}$ --- not some multiple of $u_{\rho / \rho'}$, which means that this generator can be taken as a representative in all calculations.
\par The maximal cones of the affine Bergman fan are the following:
$$\langle e_{F_1}, \ldots, e_{F_{r-1}}, e_{F_{r}} = (1, \ldots, 1) \rangle, \langle e_{F_1}, \ldots, e_{F_{r-1}}, -e_{F_{r}} = (-1, \ldots, -1) \rangle,$$
and the weights are all equal to 1 (slightly abusing notations, we now mean by $e_F$'s actual characteristic vectors in $\rr^E$, not their classes in the quotient $\rr^E / \langle (1, \ldots, 1) \rangle$). Thus, one can check that the balancing condition for any $(r-1)$-dimensional cone containing the all-ones or minus all-ones boils down to the third axiom (the covering axiom) of Definition \ref{def:matroid}; and the balancing condition for a $(r-1)$-dimensional cone not containing the all-ones or minus all-ones is trivially satisfied, as any such cone is contained in just two $r$-dimensional cones which cancel each other out in the balancing condition.
\par It is convenient to consider the \textit{refinements} of tropical fans (\cite{ARau}, Definition 2.8), which are essentially the tropical fans with the same support and the same weight for the interior points of maximal cones. This leads to
\begin{definition}[\cite{ARau}, Definitions 2.12, 2.15]
A \textit{tropical cycle} is a class of equivalence of tropical fans with respect to equivalence relation ``to have a common refinement''. A \textit{tropical subcycle} $Z$ of the tropical cycle $X$ is a tropical cycle whose support is a subset: $|Z| \subset |X|$.
\end{definition}
Apart from the very general Definition \ref{def:tropcor}, all the cones in the subcycles will be the faces of the simplicial maximal cones of the fans they lie in.
\par 
Considering tropical cycles allows to define new tropical cycles without explicitly listing all the cones, but describing the support and the weights of the interior points of the maximal cones instead. This is useful to us when introducing products and stars.
\begin{definition}
The \textit{product} of tropical fans $X \in V_X$ and $Y \in V_Y$ is a tropical cycle $X \times Y \in V_X \times V_Y$ such that $|X \times Y| = |X| \times |Y|$, and the weight of $(x,y) \in X \times Y$ is equal to the product of $x \in X$ and $y \in Y$ for $x,y$ interior points of the maximal cones of $X,Y$.
\end{definition}
It is known that the tropical cycles $B(M \oplus N)$ and $B(M) \times B(N)$ coincide (\cite{FRau}, Lemma 2.1), but the fans are different --- $B(M \oplus N)$ is a refinement of $B(M) \times B(N)$ (see also Lemma \ref{lemma:ffproduct}).
\begin{definition} \label{def:star}
Let $X \subset \rr^n$ be a tropical fan and $p \in \rr^n$ a point (not necessarily in $X$). The \textit{star} of $p$ in $X$ denoted by $\mathrm{Star}_X(p)$ is a tropical fan with support
$$\{v \in \rr^n \colon \exists \; c>0 \; \forall \; 0<c'<c \; \; p + c' \cdot v \in |X| \}$$
and the weight of $v$ equal to the weight of $p + c' \cdot v$ for interior points of maximal cones. In particular, the star is empty if and only if $p$ does not belong to the support of $X$.
\end{definition}
If $B(M)$ is a Bergman fan, and $p$ is an interior point of the cone of any dimension generated by the flag $\varnothing = F_0 \subsetneq F_1 \subsetneq \ldots F_k = E$ of flats $F_i \in \mathcal{F}(M)$, then the explicit description of the $\mathrm{Star}_{B(M)}(p)$ is known (see, for example, \cite{FRau}, Lemma 2.2):
$$B(M|_{F_1}/F_0) \times B(M|_{F_2}/F_1) \ldots \times B(M|_{F_k}/F_{k-1}).$$
\par The crucial ingredient of the tropical intersection theory is the possibility to cut out Weil divisors of what is called ``rational functions'' on tropical fans.
\begin{definition}[\cite{ARau}, Section 3] \label{def:Weil}
A \textit{rational function} on the tropical fan is a continuous function on the union of its maximal cones which is linear on each of them. A \textit{Weil divisor} of the rational function $\phi$ on the $k$-dimensional tropical fan $X$ is a $(k-1)$-dimensional tropical subfan denoted by $\phi \cdot X$ with the weight on the $(k-1)$-dimensional cone $\tau$ of $X$ given by
$$\omega_{\phi \cdot X} (\tau) = \sum_{\sigma \supset \tau} \omega_X (\sigma) \phi(v_{\sigma / \tau}) - \phi \left( \sum_{\sigma \supset \tau} \omega_X (\sigma) v_{\sigma / \tau} \right),$$
where $v_{\sigma / \tau}$ is any representative of the primitive vector $u_{\sigma / \tau}$ in the ambient space. One shows that this definition does not depend on the choice of representatives.
\par More elegant, but not as useful in calculations, is the equivalent definition of associated Weil divisor --- that it is the $\rr^n$-section of the only way to ``balance'' the graph of $\phi$ in $\rr^{n+1} = \rr^n \times \rr$, i.e., to add cones containing direction along the last coordinate corresponding to the value of $\phi$ with some weights, so that the disbalance brought by the non-linearity of $\phi$ on the hyperface junctions of maximal cones is fixed. Thus, the divisor has zero weights if the function is globally linear (mind that the converse is not true, but the counter-examples are not going to occur in this paper).
\end{definition}
In the general setting, it is common to take advantage of the opportunity to refine the fan, then take the fine rational function which is not conewise linear before the refinement and consider this to be a rational function on the tropical cycle given by the equivalence class of the fan. In our setting, though, values on the primitive vectors of the generating rays of cones (extended conewise linearly) always suffice for the functions at play; so we do not concern ourselves with the details and refer for them to \cite{ARau}. We may also sometimes abuse the notation and say ``value on F'' meaning value on the primitive vector of the one-dimensional cone $\langle e_F \rangle$ which is $e_F$ itself.
\par Finally, because of the technical necessity to work with affine fans, all our functions are assumed to be linear on the \textit{lineality space} $\langle (1 \ldots, 1) \rangle$, which means that $\phi((-1 \ldots, -1)) = -\phi((1 \ldots, 1))$. This guarantees that all the maximal cones of the Weil divisors contain either $e_E$ or $-e_E$ as well. Henceforth, when verifying something about the fan, we do it only for positive cones, assuming that negative cones are treated exactly the same.
\par The most ubiquitous rational function $\alpha$ is described in the following
\begin{definition} \label{def:truncation}
A \textit{truncation} of the Bergman fan $B(M)$ is the Weil divisor of the function $\alpha$ which is equal to $-1$ on $E$ and $0$ on every other ray. It can be shown (a simple partial case of Lemma \ref{modcuts} ahead) that the resulting fan is again a Bergman fan of the \textit{truncated} matroid, denoted by $\mathrm{\mathop{Tr}}(M)$, whose flats are the same as of $M$ except for the hyperplanes.
\end{definition}
Cutting out a divisor of $\alpha$ is analogous in tropical geometry to taking a general hyperplane section in algebraic geometry, and thus
\begin{definition}
The \textit{degree} of the $k$-dimensional affine tropical fan $X$ is the weight of the cone $\langle e_E \rangle$ in the fan $\underbrace{\alpha \cdots \alpha}_{k-1} \cdot X$.
\end{definition}
It is easy to see that Bergman fans have degree 1. Moreover,
\begin{theorem}[\cite{Fink}, Theorem 6.5] \label{th:Fink}
Tropical fan of degree 1 with positive weights is a Bergman fan of matroid.
\end{theorem}
Theorem \ref{th:Fink} was proven in \cite{Fink} in more general setting, for tropical varieties, not only for fans, and may be considered an overkill for our needs, where we could use techniques developed in Section \ref{section:flagfans} instead, but we believe it is a beautiful result, and we are happy to apply it when relevant. Note that without requiring the weights to be positive this is not true, of course, as we are going to see many times.
\par To define tropical fibre product and to work with tropical intersections effectively, it remains to recall the definition of tropical morphism, the constructions of the pullback and the push-forward.
\begin{definition}[\cite{ARau}, Definition 4.1]
A \textit{morphism} $f \colon X \to Y$ between tropical fans $X \subset \zz^m \otimes_{\zz} \rr$ and $Y \subset \zz^n \otimes_{\zz} \rr$ is a map from the support of $X$ to the support of $Y$ induced by a $\zz$-linear map $\tilde{f} \colon \zz^m \to \zz^n$.
\end{definition}
In full generality, the maximal cone $\sigma$ of $X$ maps to some subset $f(\sigma) \subset \sigma'$ of $Y$. For instance, the embedding of the refinement into the unrefined fan is a tropical morphism. For tropical morphisms we consider, though, such thing is not going to happen, as generating rays of cones of $X$ always map to the generating rays of cones of $Y$. This observation also makes the next definitions simpler.
\begin{definition}[\cite{ARau}, Proposition 4.7]
Let $f \colon X \to Y$ be a morphism of tropical fans, and let $\phi$ be a rational function on $Y$ with $\phi \cdot Y$ its Weil divisor. Then the \textit{pullback} of $\phi$ is a function $f^*(\phi) \colon X \to \rr$ defined as $f^*(\phi)(x) = \phi(f(x))$, and the pullback of the Weil divisor $\phi \cdot Y$ is the Weil divisor $f^*(\phi) \cdot X$.
\par One shows that pullback is a linear map from divisors on $Y$ to divisors on $X$.
\end{definition}
\begin{definition} \label{def:pushforward}
Let $f \colon X \to Y$ be a morphism of tropical fans lying in spaces $\zz^m \otimes_{\zz} \rr$ and $\zz^n \otimes_{\zz} \rr$, respectively. Let $Z \subset X$ be a subcycle. Then, the \textit{push-forward} $f_*(Z)$ is defined as a tropical subcycle of $Y$ consisting of maximal cones $f(\sigma)$ for $\sigma$ the cone of $X$ if $f$ is injective on $\sigma$ (and their faces). The weight on the maximal cone should be defined by the formula
$$\omega_{f_*Z}(\sigma')=\sum_{f(\sigma)=\sigma'} \omega_Z(\sigma) \cdot \lambda(\sigma',\sigma),$$
where $\lambda(\sigma',\sigma)$ is the index of the sublattice $f(V_{\sigma} \cap \zz^m)$ in the lattice $f(V_{\sigma'} \cap \zz^n)$ (in the notation of Definition \ref{def:tropicalfan}).
\end{definition}
Fortunately, again, in our setting the integer lattice map is always surjective, therefore, the index of the sublattice is always $1$. The sum in the definition of the weight is important, though, as there can be several maximal cones of $X$ mapping to the same maximal cone of $Y$.
\par Now we are almost ready to recall the main notion for Theorem \ref{th:tfp} --- the tropical fibre product. Its only ingredient not covered so far is the construction of ``diagonal'' rational functions from Definition \ref{diagconstr}, which arise as the partial case of more general construction from Lemma \ref{modcuts}. For the clarity of exposition, the definition is postponed to Section \ref{section:flagfans}, so the reader can either look at functions $\phi_i$ as a black box for now, or comprehend their definition as a separate construction.
\begin{definition}[\cite{FHampe}, Definition 3.6] \label{def:tfp}
Let $M_1$ and $M_2$ be matroids on groundsets $E_1$ and $E_2$, respectively. Let $N$ be their common restriction, i.e., a matroid on groundset $T = E_1 \cap E_2$ such that $M_1|_{T} = N = M_2|_{T}$. Let $\pi_1 \colon B(M_1) \to B(N)$ and $\pi_2 \colon B(M_2) \to B(N)$ be projections induced by the embeddings $T \hookrightarrow E_1$ and $T \hookrightarrow E_2$. Let $\phi_1 \cdots \phi_r \cdot B(N) \times B(N) = \Delta \subset B(N) \times B(N)$, like in Definition \ref{diagconstr}, and let $\pi = \pi_1 \times \pi_2$. Then, the \textit{tropical fibre product} $B(M_1) \times_{\pi_1, B(N), \pi_2} B(M_2)$ is a cycle in $B(M_1) \times B(M_2)$ defined as
$$\pi^*(\phi_1) \cdots \pi^*(\phi_r) \cdot B(M_1) \times B(M_2) \subset B(M_1) \times B(M_2).$$
\end{definition}
\begin{example} \label{ex:tfps}
We consider three instances of the tropical fibre product to familiarize the reader with the notion:
\begin{enumerate}
\item Assume $T$ is a modular flat of one of the matroids, say, $M_1$. In this case, $\pi_1 \colon B(M_1) \to B(N)$ is not only surjective, but \textit{locally surjective}, a term defined in \cite{FHampe}. It means that for any $p \in |B(M_1)|$, the restriction of $\pi_1$ on some neighborhood of $p$ is surjective onto some neighborhood of $\pi_1(p) \in |B(N)|$. This, in its turn, is equivalent to the projection $\pi_1$ being surjective on $\mathrm{Star}_{B(N)}(\pi_1(p))$. By description of stars of Bergman fans (paragraph after Definition \ref{def:star}) we need to show that for any pair of flats $F'' \supset F'$ of $M_1$ the restriction map 
$$\mathcal{F}(M_1|_{F''}/F') \to \mathcal{F}(M_1|_{T \cap F''}/(T \cap F')), \; F \mapsto F \cap T$$
is surjective. We show that for the flat $F_N$ of $N$ containing $F' \cap T$ the preimage is, for example, $\cl_{M_1}(F' \cup F_N)$. Indeed, using modularity of $T$,
$$\rk_{M_1}(F' \vee F_N) = \rk_{M_1}{F'} + \rk_{M_1}{F_N} - \rk_{M_1}{F' \cap T},$$
and then
\begin{equation*}
\begin{split}
\rk_{M_1}((F' \vee F_N) \cap T) = \rk_{M_1}(F' \vee F_N) + \rk_{M_1}{T} - \rk_{M_1}{F' \vee T} = \rk_{M_1}{F'} + \\ + \rk_{M_1}{F_N} - \rk_{M_1}{F' \cap T} + \rk_{M_1}{T} - \rk_{M_1}{F'} - \rk_{M_1}{T} + \rk_{M_1}{F' \cap T} = \rk_{M_1}{F_N},
\end{split}
\end{equation*}
so $(F' \vee F_N) \cap T = F_N$ as claimed.
In Theorem 3.9 of \cite{FHampe} it is shown that when one of the projections is locally surjective, the tropical fibre product is a tropical fan with positive weights. Moreover, its support is equal to the set-theoretic fibre product of $B(M_1)$ and $B(M_2)$ over $B(N)$.
\begin{figure}[h]
\includegraphics[width=\textwidth]{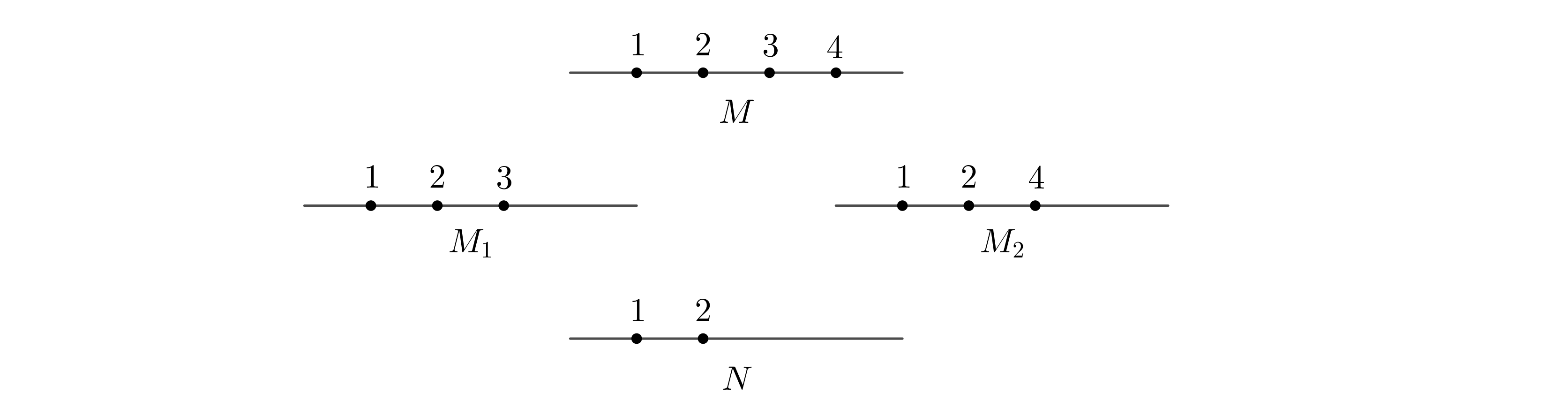}
\caption{$B(M)$ is the tropical fibre product $B(M_1) \times_{\pi_1, B(N), \pi_2} B(M_2)$.}
\label{fig:1234}
\end{figure}
\item Consider $M_1, M_2, N$ as shown on Figure \ref{fig:1234}. In this case one verifies that the tropical fibre product coincides with $B(M)$ (the elements $1$ and $2$ of $M$ should be thought of as pairs of respective parallel elements of $M_1$ and $M_2$ --- more precisely in Theorem \ref{th:tfp}). Observe that $M$ is the proper amalgam of $M_1$ and $M_2$. Mind that the set-theoretic product $B(M_1) \times_{\pi_1, B(N), \pi_2} B(M_2)$ is not even a pure-dimensional fan and, therefore, certainly not tropical. It happens because, unlike part (1) of this example, the projections $B(M_1) \to B(N)$ and $B(M_2) \to B(N)$ are not locally surjective on the rays corresponding to flats $\{3\}$ and $\{4\}$ respectively. Projective fans illustrating this phenomenon are shown on Figure \ref{fig:1234-fan}.
\begin{figure}[h]
\includegraphics[scale=0.25]{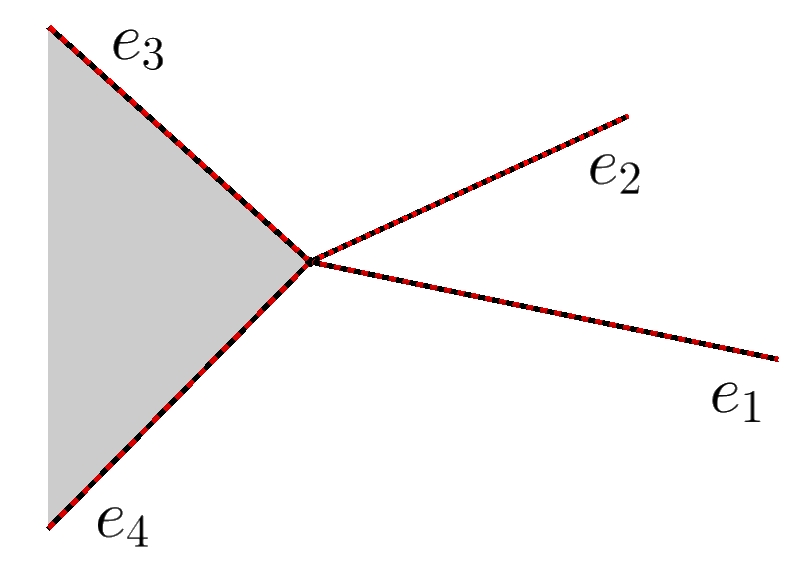}
\caption{Set-theoretic fibre product is in black and gray, $B(M)$ is in red.}
\label{fig:1234-fan}
\end{figure}
\item Consider $M_1,M_2,N$ as in Example \ref{ex:cross}. Then, it can be calculated that, after taking two out of three Weil divisors, $\pi^*(\phi_2) \cdot \pi^*(\phi_1) \cdot (B(M_1) \times B(M_2)) = B(M')$, where $M'$ is shown on Figure \ref{fig:crossspace}. Then, $\pi^*(\phi_3)(\{1_1,1_2,3_1,3_2,5,6\}) = \pi^*(\phi_3)(\{2_1,2_2,4_1,4_2,5,6\}) = -1$, but $\pi^*(\phi_3)(\{5,6\})=0$, despite $\{1_1,1_2,3_1,3_2,5,6\}$ and $\{2_1,2_2,4_1,4_2,5,6\}$ being a modular pair. Thus, $\pi^*(\phi_3)$ is not defined by the modular cut of $M'$ (see Lemma \ref{modcuts}), and the tropical fibre product has two cones with weight $-1$, corresponding to the flags $\varnothing \subset \{5\} \subset \{5,6\} \subset E_1 \sqcup E_2$ and $\varnothing \subset \{6\} \subset \{5,6\} \subset E_1 \sqcup E_2$.
\begin{figure}[h]
\includegraphics[width=\textwidth]{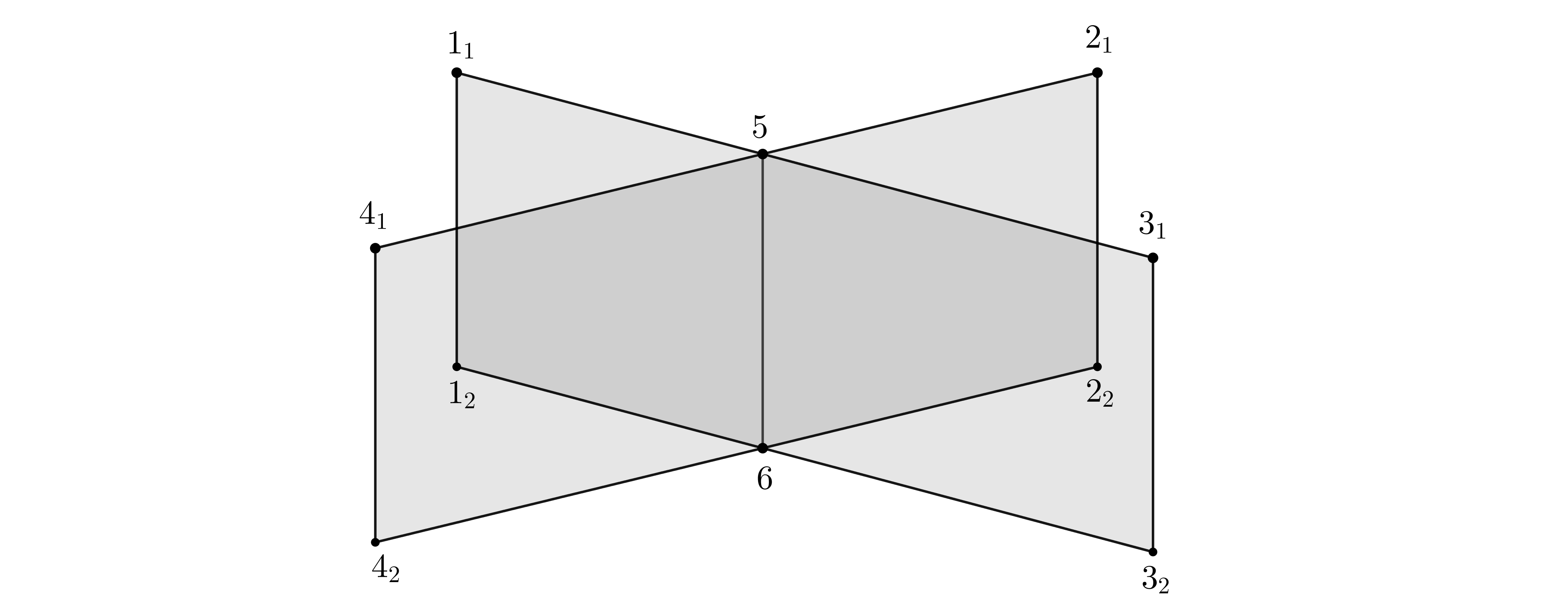}
\caption{Indices discern between elements of $E_1$ and $E_2$ in $E_1 \sqcup E_2$.}
\label{fig:crossspace}
\end{figure}
\end{enumerate}
\end{example}
\begin{lemma} \label{trres=restr}
Restriction commutes with truncation (preserving dimension). More precisely, let $M$ be a matroid on groundset $E$ with $\rk{M} = r$. let $T \subset E$ and let $M|_T=N$ with $\rk N = s$. Let $\pi \colon B(M) \to B(N)$ be a projection induced by the embedding $T \hookrightarrow E$. Let $\alpha_M$ be a truncation rational function in $\rr^E$ defined in \ref{def:truncation}, and let $Z_M = \alpha_M^{r-i} \cdot B(M)$ be an $i$-dimensional truncation of $B(M)$. Similarly, let $\alpha_N$ be a truncation function in $\rr^T$, and let $Z_N = \alpha_N^{s-i} \cdot B(N)$ be an $i$-dimensional truncation of $B(N)$. Then, $\pi_*(Z_M)=Z_N$.
\end{lemma}
\begin{proof}
Since $Z_M = B(\mathop{\mathrm{Tr}}^{r-i}(M))$, maximal cones of $Z_M$ are generated by characteristic rays of flags of the form $\varnothing \lessdot F_1 \lessdot \ldots \lessdot F_{i-1} \leqslant F_i = E$. The image of such a cone under $\pi$ is $i$-dimensional if and only if all $F_i \cap T$ are different flats of $B(\mathop{\mathrm{Tr}}^{s-i}(N))$. Any such flag can be restored from the flag $\varnothing \lessdot F'_1 \lessdot \ldots \lessdot F'_{i-1} \leqslant F'_i = T$ of flats in $N$ by defining $F_i = \cl_M{F'_i}$. Thus, there is a bijection between the maximal cones of $Z_N$ and maximal cones of $Z_M$ with zero intersection with $\ker{\pi}$. All the indices of the respective lattices are equal to $1$, as always, as each lattice is generated by the primitive vectors of the generating rays.
\end{proof}
Finally, we will need a few more statements providing tying tropical intersections in related objects --- the domain and the range of the tropical morphism; or the tropical fan and its star.
\begin{theorem}[Projection formula, \cite{ARau}, Proposition 4.8] \label{th:projformula}
Let $f \colon X \to Y$ be a morphism of tropical fans. Let $C$ be a subcycle of $X$, and $\phi$ a rational function on $Y$. Then,
$$\phi \cdot f_*(C) = f_*(f^*(\phi) \cdot C).$$
\end{theorem}
If $\phi$ is a rational function on $X$, then the rational function $\phi^{p}$ on the star $\mathrm{Star}_X(p)$ is defined as the linear extension of the restriction of $\phi$ onto the small neighborhood of $p$.
\begin{lemma} \label{lemma:stardivisor}
Taking the star commutes with cutting out Weil divisor. More precisely, if $X$ is a tropical fan, $p$ a point and $\phi$ a rational function, then
$$\mathrm{Star}_{\phi \cdot X}(\rho) = \phi^{\rho} \cdot \mathrm{Star}_X(\rho).$$
\end{lemma}
\begin{lemma}[\cite{FRau}, Definition 8.1] \label{setdiag}
Given morphism of Bergman fans $f \colon X \to Y$, if $Z \subset Y$ is a subcycle obtained from taking Weil divisors of functions $\phi_i$, then the support of $\prod f^*(\phi_i) \cdot X$ is contained in $f^{-1}(|Z|)$.
\end{lemma}
\begin{remark}
Let us explain Lemma \ref{setdiag} via a more general setting of intersection theory in Bergman fans. In \cite{FRau}, Section 4, the intersection $C \cdot D$ of arbitrary subcycles of Bergman fans is defined using diagonal construction from Definition \ref{diagconstr} below. It is shown there in Theorem 4.5 that the support of $C \cdot D$ is contained in the intersection of the supports of $C$ and $D$. Then, in Definition 8.1, a pullback of the arbitrary cycle $D \subset B(N)$ with respect to the morphism $f \colon B(M) \to B(N)$ is defined as the intersection in $B(M) \times B(N)$ of the cycle $\Gamma_f$ --- the graph of $f$ defined as the push-forward; and the cycle $B(M) \times D$; this intersection then push-forwarded onto the first factor $X$ of $X \times Y$. This construction of the pullback coincides with the one we use for those cycles $D \subset B(N)$ that are cut out by rational functions (see Example 8.2 of \cite{FRau}). Hence, the comment after Definition 8.1 of \cite{FRau} claims that the support of $f^*D$ is contained in the preimage of the support of $D$.
\end{remark}
\section{Flag fans} \label{section:flagfans}
In this section we introduce the generalization of Bergman fan which we call a Flag fan. This class of tropical fans is stable under taking Weil divisors with respect to certain rational functions which we call (simply, not modular) cut functions.
\par
We begin by recalling the standard setting of modular cuts and tropical modifications. The following lemma is a summary of results established at approximately the same time by Shaw, Francois and Rau in \cite{Kris} and \cite{FRau}.
\begin{lemma} \label{modcuts}
Let $M$ be a matroid, $B(M)$ its Bergman fan, $\mathcal{M} \subset \mathcal{F}(M)$ a modular cut. Define rational function 
$$\phi_{\mathcal{M}}(e_F) = \begin{cases}
-1, & \mbox{if $F \in \mathcal{M}$}, \\
0, & \mbox{otherwise.}
\end{cases}$$
Assume $\mathcal{M}$ does not contain flats of rank 1. Then $\phi_{\mathcal{M}} \cdot B(M) = B(N)$, where $F \in \mathcal{F}(N)$ if and only if
$$\begin{cases}
F \in \mathcal{F}(M) \\
\left[ \begin{array}{ll}
F \in \mathcal{M} \\
F' \notin \mathcal{M} \mbox{ for any } F' \gtrdot F 
\end{array} \right.
\end{cases}$$
Moreover, the balanced graph of $\phi_{\mathcal{M}} \colon B(M) \to \rr$, denoted by $\Gamma_{\phi} \in \rr^E \times \rr$, which is the union of the set-theoretic graph and cones containing generator along the last coordinate (see Definition \ref{def:Weil}), is equal to the Bergman fan $B(M')$, where $M'$ is a one-element extension of $M$ on the groundset $E \cup \{e\}$, and $F \cup \{e\} \in \mathcal{F}(M')$ if and only if $F \in \mathcal{M}$. Matroid $N$ is, in its turn, the contraction of the element $e$ of $M'$ (see Definition \ref{def:matroid}).
\par In the other direction, if $|B(N)| \subset |B(M)|$ is a subfan of codimension $1$, then there exists a modular cut $\mathcal{M}$ such that $\phi_{\mathcal{M}} \cdot B(M) = B(N)$. Function $\phi_{\mathcal{M}}$ can be defined on the rays of $B(M)$ as
$$\phi_{\mathcal{M}}(e_F) = \rk_{N}{F} - \rk_{M}{F}.$$
More generally, if $|B(N)| \subset |B(M)|$ is a subfan of codimension $s$, then there exists a sequence of functions $\phi_i = \phi_{\mathcal{M}_i}$ corresponding to the modular cuts $\mathcal{M}_i$ such that $\prod_{i=1}^{s}{\phi_{i}} \cdot B(M) = B(N)$. Mind that $\mathcal{M}_j$ is a modular cut of $\prod_{i=1}^{j-1}{\phi_{i}} \cdot B(M)$, not necessarily of $B(M)$. Functions $\phi_i$ can be defined on the rays of $B(M)$ as
\begin{equation} \label{eq:cutfunctions}
\phi_i(e_F) = \begin{cases}
0, & \mbox{if $\rk_{N}(F)+s-i \geqslant \rk_{M}(F)$} \\
-1, & \mbox{otherwise.}
\end{cases}
\end{equation}
In other words,
$$\sum_{i=1}^s \phi_{i}(F) = \rk_{N}{F} - \rk_{M}{F}, \; \mbox{ and } \; -1 \leqslant \phi_{i}(F) \leqslant \phi_{j}(F) \leqslant 0 \; \mbox{ for } \; i>j.$$
\end{lemma}
\begin{remark}
As shown in \cite{ARau}, Proposition 3.7 (b), the order of taking Weil divisors with respect to functions $\phi_i$ does not matter. We are going to stick to the non-increasing value order, though, as it is easier to track the result. For example, if, say, $\phi_2$ is not a characteristic function of a modular cut of the initial matroid $M$, but of a modular cut on $\phi_1 \cdot B(M)$, then $\phi_1 \cdot \phi_2 \cdot B(M) = \phi_2 \cdot \phi_1 \cdot B(M)$ has positive weights, while $\phi_2 \cdot B(M)$ does not.
\end{remark}
\par The next definition we recall is the partial case of Lemma \ref{modcuts} and the main advancement of \cite{FRau} --- cutting out the set-theoretic diagonal allowed to define intersection of arbitrary subcycles of Bergman fans.
\begin{definition}[\cite{FRau}, Corollary 4.2] \label{diagconstr}
Given a Bergman fan $B(N)$, the functions $\phi_{N,i}=\phi_i$ cutting the diagonal subcycle $\Delta \subset B(N) \times B(N)$ are the piecewise linear functions on $B(N) \times B(N)$ such that $\phi_1 \cdots \phi_r \cdot B(N) \times B(N) = \Delta \subset B(N) \times B(N)$. They are defined by their values on rays $F = F_1 \sqcup F_2$ by
$$\phi_i(e_F) = \begin{cases}
0, \mbox{ if $\rk_{N}(F_1 \cup F_2) + r - i \geqslant \rk_{N}(F_1)+\rk_{N}(F_2)$}  \\
-1, \mbox{ otherwise.}
\end{cases}$$
\end{definition}
Let us now relax the requirements for the function by cancelling the modularity condition.
\begin{definition} 
Consider the pair $(\mathcal{G},\omega_{\mathcal{G}})$ that consists of the ranked poset $\mathcal{G}$ of subsets of $E$ ($\varnothing,E \in \mathcal{G}$, also if $X \leqslant_{\mathcal{G}} X'$, then $X \subset X'$, but not vice versa) and of the \textit{edge weight} function $w_\mathcal{G} \colon \{(F,F') \in \mathcal{G}| \; F \lessdot_{\mathcal{G}}  F\} \to \zz$.
Then the \textit{Flag fan} corresponding to the pair $(\mathcal{G},\omega_{\mathcal{G}})$ is a tropical fan in $\rr^E$ such that each cone is generated by the characteristic vectors of the set of pairwise comparable elements of $\mathcal{G}$ (which we call a flag, as for flats of matroids), and the weight on the maximal cones is given by
$$w(F_0 \lessdot  \ldots \lessdot  F_r) = \prod_{i=0}^{r-1} w_\mathcal{G}(F_i, F_{i+1}).$$
As usual, the copies of cones with $e_E$ replaced with $-e_E$ are added with the same weights.
\end{definition}
\begin{definition}
A \textit{cut} on a ranked poset $\mathcal{G}$ is a subset of flats $\mathcal{A}$ such that if $F' \supset F \in \mathcal{A}$, then $F' \in \mathcal{A}$. A \textit{cut function} $\phi_{\mathcal{A}}$ on the Flag fan of $\mathcal{G}$ is a piecewise linear continuation of $-\chi_{\mathcal{A}}$, a function equal to $-1$ on the rays $e_F$ for $F \in \mathcal{A}$ and equal to $0$ on the other rays (except, as always, $-e_E$, where the value must be linear extension of $\phi_{\mathcal{A}}(e_E)$).
\end{definition}
\begin{example} Here are a few examples of Flag fans to get used to them:
\begin{enumerate}
\item Bergman fan $B(M)$ is a Flag fan corresponding to the pair $(\mathcal{F}(M),\omega_{\mathcal{F}(M)})$, where \newline $\omega_{\mathcal{F}(M)}(F,F')=1$ for any $F \lessdot_{M} F'$.
\item Consider $\mathcal{G}$ consisting of $\varnothing, \{1,2,3\}$ and sets $\{1\}, \{1,2\}, \{1,3\}$ all covering $\varnothing$ and all covered by $\{1,2,3\}$. Note that $\{1\} \subset \{1,2\}, \{1,3\}$, but they all have rank $1$ in $\mathcal{G}$. Let all $\omega_{\mathcal{G}}(X,X')=1$ except for $\omega_{\mathcal{G}}(\{1\},\{1,2,3\})=-1$. It is easy to verify that the balancing condition is satisfied, so that we indeed get a tropical fan. Besides being the simplest non-Bergman Flag fan, it is nice to keep in mind because a very similar fan will appear naturally in Example \ref{ex:simpletrcor}.
\item A tropical fibre product from Example \ref{ex:tfps} (3) is a Flag fan. The Hasse diagram of its poset $\mathcal{G}$ is shown on Figure \ref{fig:hasse}, and all the weights are $1$ except for $\omega_{\mathcal{G}}(\{5,6\}, E_1 \sqcup E_2) = -1$.
\begin{figure}[h]
\includegraphics[width=\textwidth]{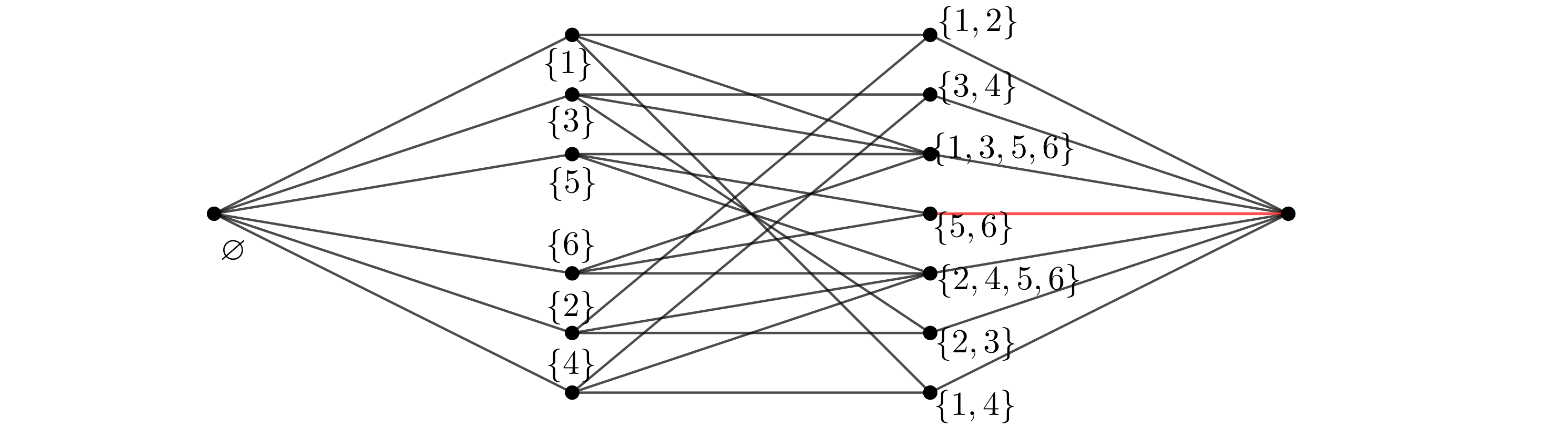}
\caption{The only covering edge with negative weight is in red.}
\label{fig:hasse}
\end{figure}
\end{enumerate}
\end{example}
\begin{lemma} \label{lemma:ff}
If $X$ is a Flag fan and $\phi = \phi_{\mathcal{A}}$ a cut function on $X$, then $Y = \phi_{\mathcal{A}} \cdot X$ is a Flag fan.
\end{lemma}
\begin{proof}
The idea of proof is straightforward: we take Definition \ref{def:Weil} and verify the weights on the hyperfaces of maximal cones of $X$. Maximal cone $\sigma_{\mathbf{F}}$ corresponds to the maximal flag $\mathbf{F}$ of $\mathcal{G}$ consisting of $\varnothing = F_0 \lessdot  F_1 \lessdot  \ldots \lessdot  F_r = E$ with $F_i \in \mathcal{G}$, and each of its hyperfaces corresponds to $\mathbf{F}$ with one missing set. Since $\phi$ is a monotonously non-increasing function, and sets $F_i$ are pairwise comparable, there exists unique $i$ such that $\phi(F_i) = 0$ and $\phi(F_{i+1}) = -1$. Consider $\tau = \tau_{\mathbf{F}'}$ where $\mathbf{F}' = \mathbf{F} \setminus \{F_j\}$. Recall that the unique generator of the maximal cone $\sigma$ not belonging to the face $\tau$ can be chosen as $v_{\sigma / \tau}$. Denote by $\mathbf{F}_{F \to F'}$ the flag $\mathbf{F}$ with $F$ replaced by $F'$. We get
\begin{equation*}
\begin{split}
\omega_{Y} (\tau) = \sum_{\sigma \supset \tau} \omega_X (\sigma) \phi(v_{\sigma / \tau}) - \phi \left( \sum_{\sigma \supset \tau} \omega_X (\sigma) v_{\sigma / \tau} \right) = \\ = \sum_{F_{j+1}\gtrdot F'\gtrdot F_{j-1}} \omega_X(\sigma_{\mathbf{F}_{F_j \to F'}})\phi(e_{F'}) - \phi \left(\sum_{F_{j+1}\gtrdot F'\gtrdot F_{j-1}} \omega_X(\sigma_{\mathbf{F}_{F_j \to F'}}) e_{F'}\right) = \\ =
\sum_{F_{j+1}\gtrdot F'\gtrdot F_{j-1}} \prod_{k \neq j-1,j} w_\mathcal{G}(F_k, F_{k+1}) \cdot w_\mathcal{G}(F_{j-1}, F') \cdot w_\mathcal{G}(F', F_{j+1}) \cdot \phi(e_{F'}) - \\ -
\phi \left(\sum_{F_{j+1}\gtrdot F'\gtrdot F_{j-1}} \prod_{k \neq j-1,j} w_\mathcal{G}(F_k, F_{k+1}) \cdot w_\mathcal{G}(F_{j-1}, F') \cdot w_\mathcal{G}(F', F_{j+1}) \cdot e_{F'} \right) = \\ =
\prod_{k \neq j-1,j} w_\mathcal{G}(F_k, F_{k+1}) \left( \sum_{F_{j+1}\gtrdot F'\gtrdot F_{j-1}} w_\mathcal{G}(F_{j-1}, F') \cdot w_\mathcal{G}(F', F_{j+1}) \cdot \phi(e_{F'}) \right. - \\ -
\phi \left. \left(\sum_{F_{j+1}\gtrdot F'\gtrdot F_{j-1}} w_\mathcal{G}(F_{j-1}, F') \cdot w_\mathcal{G}(F', F_{j+1}) \cdot e_{F'} \right) \right) = \\ =
\prod_{k \neq j-1,j} w_\mathcal{G}(F_k, F_{k+1}) \cdot \deg{(\phi|_{X_{F_{j-1},F_{j+1}}} \cdot X_{F_{j-1},F_{j+1}})},
\end{split}
\end{equation*}
where $X_{F_{j-1},F_{j+1}}$ in $\rr^{F_{j+1} \setminus F_{j-1}}$ is a Flag fan of the segment $[F_{j-1},F_{j+1}] \subset \mathcal{G}$ with induced edge weights. Thus, the weight $\omega_Y(\tau)$ can be calculated locally --- besides weights of covers of $\sigma$, it depends only on the edge weights of $[F_{j-1},F_{j+1}]$ and values of $\phi$ on that segment of $\mathcal{G}$.
\par The first thing we notice is that if a hyperface $\tau = \tau_{\mathbf{F}'}$ contains both $F_i$ and $F_{i+1}$ --- that is, the missing level is neither $i$-th nor $(i+1)$-st (say, it is $j$-th), then $\omega_{Y}(\tau) = 0$. Indeed, in this case $\phi|_{X_{F_{j-1},F_{j+1}}}$ is constant 0 or constant $-1$, therefore linear, not just piecewise linear, so its Weil divisor is zero.
\par Note that the modified pair $(\mathcal{G}_{\phi}, \omega_{\mathcal{G}_{\phi}})$, where $\mathcal{G}_{\phi}$ has the same covering relations $\gtrdot $ with the same weights as $\mathcal{G}$, except those with different values of $\phi$, and additional covers between $F_{j-1} \lessdot _{\mathcal{G}_{\phi}} F_{j+1}$ if $\deg{(\phi|_{X_{F_{j-1},F_{j+1}}} \cdot X_{F_{j-1},F_{j+1}})} \neq 0$ with those non-zero weights, is also ranked: $\rk_{\mathcal{G}_{\phi}}(F) = \rk_{\mathcal{G}}(F) + \phi(F)$, and $Y$ is the Flag fan of this new pair. Indeed, every maximal cone of $Y$ with non-zero weight has generators with all possible ranks in $\mathcal{G}_{\phi}$ --- from $1$ to $r-1$.
\end{proof}
\begin{corollary}[Algorithmic description] \label{cor:alg}
To construct pair $(\mathcal{G}_{\phi}, \omega_{\mathcal{G}_{\phi}})$ from pair $(\mathcal{G}, \omega_{\mathcal{G}})$, perform the following steps:
\begin{enumerate}
\item Delete all covers $F \lessdot _{\mathcal{G}} F'$ if $\phi(F) \neq \phi(F')$;
\item For all segments $[F_{j-1},F_{j+1}]$ such that $\phi(F_{j-1}) \neq \phi(F_{j+1})$, add covers $F_{j-1} \lessdot _{\mathcal{G}_{\phi}} F_{j+1}$ with $\omega_{\mathcal{G}_{\phi}}(F_{j-1},F_{j+1}) = \deg{(\phi|_{X_{F_{j-1},F_{j+1}}} \cdot X_{F_{j-1},F_{j+1}})}$ if it is not 0;
\item Delete those vertices which do not lie in the resulting graph of covers in the same connected component as $\varnothing$ and $E$ (if $\varnothing$ and $E$ are in different connected components, $\phi \cdot X$ is zero).
\end{enumerate}
\end{corollary}
\begin{corollary} \label{cor:tfpff}
Tropical fibre product from Definition \ref{def:tfp} is a Flag fan.
\end{corollary}
\begin{proof}
Functions $\pi^*(\phi_i)$ are pullbacks of cut functions on $B(N) \times B(N)$ with respect to the map $\pi$ which preserves inclusion of the sets, therefore, they are also cut functions. Thus, the claim follows from Lemma \ref{lemma:ff} applied multiple times.
\end{proof}
\begin{lemma} \label{lemma:ffproduct}
If $X$ is a Flag fan corresponding to the pair $(\mathcal{G},\omega_{\mathcal{G}})$, and $Y$ is a Flag fan corresponding to the pair $(\mathcal{H},\omega_{\mathcal{H}})$, then $X \times Y$ is a Flag fan corresponding to the pair $(\mathcal{G} \times \mathcal{H},\omega_{\mathcal{G} \times \mathcal{H}})$, where
$$\omega_{\mathcal{G} \times \mathcal{H}}((X_1,Y),(X_2,Y)) = \omega_{\mathcal{G}}(X_1,X_2), \;
\omega_{\mathcal{G} \times \mathcal{H}}((X,Y_1),(X,Y_2)) = \omega_{\mathcal{H}}(Y_1,Y_2).$$
\end{lemma}
\begin{proof}
The cone corresponding to the pair of maximal flags $G_i$ in $\mathcal{G}$ and $H_i$ $\mathcal{H}$ is the union of all the cones corresponding to the flags in $\mathcal{G} \times \mathcal{H}$ of the following form:
$$(G_0,H_0) \subsetneq (G_{i_1}, H_{j_1}) \subsetneq \ldots \subsetneq (G_{i_{k}}, H_{j_k}),$$
where $k = \rk{\mathcal{G}}+\rk{\mathcal{H}}$, and $(i_s,j_s) = (i_{s-1},j_{s-1}) + (1,0)$ or $(i_s,j_s) = (i_{s-1},j_{s-1}) + (0,1)$. Thus the supports coincide, and the weight of each such maximal cone  is equal to the product of the respective weights of the maximal cones in the factors, since each covering edge of both $\mathcal{G}$ and $\mathcal{H}$ is used in this flag exactly once.
\end{proof}

\section{Tropical fibre product} \label{sect:tfp}
\begin{theorem} \label{th:tfp}
Let $M_1$ and $M_2$ be matroids on groundsets $E_1$ and $E_2$, respectively. Let $N$ be their common restriction, i.e., a matroid on groundset $T = E_1 \cap E_2$ such that $M_1|_{T} = N = M_2|_{T}$. Let $\pi_1 \colon B(M_1) \to B(N)$ and $\pi_2 \colon B(M_2) \to B(N)$ be projections induced by the inclusions $T \hookrightarrow E_1$ and $T \hookrightarrow E_2$. Let $X = B(M_1) \times_{\pi_1, B(N), \pi_2} B(M_2)$ be the tropical fibre product of $B(M_1)$ and $B(M_2)$ over $B(N)$, as in Definition \ref{def:tfp}. Then:
\begin{enumerate}
\item \label{deg1} $X$ is a tropical fan of degree 1;
\item \label{easy} If there exists a proper amalgam $M$ of $M_1$ and $M_2$ over $N$, as in Definition \ref{def:pramalg}, then $B(M) \to B(M_1) \times B(M_2)$ induced by $E_1 \sqcup E_2 \to E_1 \cup E_2$ is an isomorphism on $X$;
\item \label{hard} If all weights of $X$ are positive, then the proper amalgam $M$ of $M_1$ and $M_2$ over $N$ exists.
\end{enumerate}
\end{theorem}
\begin{proof}
We begin with Claim \ref{deg1}. Let $r_1 = \rk M_1$, $r_2 = \rk M_2$, $r_0 = \rk N$. By construction, $X \subset \rr^{E_1 \sqcup E_2}$ is a tropical fan of dimension $r=r_1+r_2-r_0$, therefore, to verify that it has degree 1 we need to prove that $X \cdot H = 1 \cdot \langle e_{E_1 \sqcup E_2} \rangle$, where $H = \alpha^{r-1} \cdot B(M_1 \oplus M_2)$. By the definition of tropical fibre product, we have $X = \pi^*(\phi_1) \cdots \pi^*(\phi_{r_0}) \cdot B(M_1 \oplus M_2)$. Therefore, by commutativity of cutting out Weil divisors (\cite{FRau}, Theorem 4.5 (6)), we need $\pi^*(\phi_1) \cdots \pi^*(\phi_{r_0}) \cdot H = 1 \cdot \langle e_{E_1 \sqcup E_2} \rangle$.
\par
Observe that $\pi_*(H) = \alpha^{r_0-1} \cdot B(N \oplus N)$, since truncation commutes with restriction by Lemma \ref{trres=restr}. We also know that $\phi_1, \ldots, \phi_{r_0}$ cut out the diagonal $\Delta \subset B(N \oplus N)$, which is a Bergman fan of a matroid and hence has degree 1, so $\phi_1 \cdots \phi_{r_0} \cdot \pi_*(H) = 1 \cdot \langle e_{T_1 \sqcup T_2} \rangle$ in $B(N \oplus N)$. The desired equality then follows from consecutively applying the projection formula \ref{th:projformula}. Indeed, the first time we use it, it yields $\pi_*(\pi^*(\phi_1) \cdot H) = \phi_1 \cdot \pi_*(H)$. Next, $\pi_*(\pi^*(\phi_2) \cdot \pi^*(\phi_1) \cdot H) = \phi_2 \cdot \pi_*(\pi^*(\phi_1) \cdot H) = \phi_2 \cdot \phi_1 \cdot \pi_*(H)$, and so forth, until we get $\pi_*(\pi^*(\phi_1) \cdots \pi^*(\phi_{r_0}) \cdot H) = \phi_1 \cdots \phi_{r_0} \cdot \pi_*(H) = 1 \cdot \langle e_{T_1 \sqcup T_2} \rangle$, and the claim follows from the fact that $\langle e_{T_1 \sqcup T_2} \rangle$ is the only $1$-dimensional cone an affine fan can have, and that $\pi_*$ sums weights over the preimages of the cone.
\par
We proceed to verify Claim \ref{easy}. We will prove that $X$ coincides with the Bergman fan of the matroid $M'$ on groundset $E_1 \sqcup E_2$ which is obtained from the proper amalgam $M$ on groundset $E_1 \cup E_2$ by replacing each element of $T$ with two parallel copies. It is immediate from the definition of the Bergman fan that $B(M')$ and $B(M)$ are isomorphic with isomorphism induced by $E_1 \sqcup E_2 \to E_1 \cup E_2$. Denote the copies of $T$ in $E_1 \sqcup E_2$ by $T_1$ and $T_2$.
\par
Observe that $B(M') \subset B(M_1 \oplus M_2)$ --- since each flat of $M'$ restricts to a flat of both $M_1$ and $M_2$ (as $M_1$ and $M_2$ are restrictions of $M'$) and since flags of flats of $M'$ are also flags of flats in $M_1 \oplus M_2$. Therefore, by equation \ref{eq:cutfunctions} from Lemma \ref{modcuts}, $B(M')$ is cut by functions $\psi_1, \ldots, \psi_{r_0}$ given by
$$\psi_i(e_F) = \begin{cases}
0, & \mbox{if $\rk_{M'}(F)+r_0-i \geqslant \rk_{M_1 \oplus M_2}(F)$} \\
-1, & \mbox{otherwise.}
\end{cases}$$
We are going to show that $\psi_i$ coincides with $\pi^*(\phi_i)$ on all the rays of $\psi_{i-1} \cdots \psi_1 \cdot B(M_1 \oplus M_2)$, a Bergman fan of the matroid that we denote by $B(\tilde{M}_{i-1})$. Therefore, iterated tropical modifications with respect to functions $\psi_i$, which yield $B(M')$, and with respect to functions $\pi^*(\phi_i)$, which yield $X$, must also coincide.
\par
Consider $F \subset E_1 \sqcup E_2$ such that $e_F \subset \psi_{i-1} \cdots \psi_1 \cdot B(M_1 \oplus M_2)$, and denote 
$$X_1 = F \cap E_1; \; X_2 = F \cap E_2; \; Y_1 = F \cap T_1; \; Y_2 = F \cap T_2.$$
Then, $\rk_{M_1 \oplus M_2}(F) = \rk_{M_1}(X_1) + \rk_{M_2}(X_2)$, and
$$\rk_{M'}(F) = \min{\{\eta(Y) \colon Y \supseteq F\}} \leqslant \eta(F) = \rk_{M_1}(X_1 \cup Y_2) + \rk_{M_2}(X_2 \cup Y_1) - \rk_{N}(Y_1 \cup Y_2),
$$
with equality attained if $F$ is a flat of $M'$.
Similarly, by Definition \ref{diagconstr},
$$\pi^*(\phi_i)(e_F) = \begin{cases}
0, & \mbox{if $\rk_{N}(Y_1 \cup Y_2)+r_0-i \geqslant \rk_{N}(Y_1)+\rk_{N}(Y_2)$} \\
-1, & \mbox{otherwise.}
\end{cases}$$
Therefore, to show that $\pi^*(\phi_i)(e_F) = \psi_i(e_F)$ for all $0 \leqslant i \leqslant r_0$, we need to verify that
\begin{equation} \label{eq:x1x2y1y2}
\begin{split}
\rk_{M_1}(X_1 \cup Y_2) + \rk_{M_2}(X_2 \cup Y_1) - \rk_{N}(Y_1 \cup Y_2) - \rk_{M_1}(X_1) - \rk_{M_2}(X_2) = \\
= \rk_{N}(Y_1 \cup Y_2) - \rk_{N}(Y_1)-\rk_{N}(Y_2),
\end{split}
\end{equation}
and that $\eta(F) = \rk_{M'}(F)$, since both $\pi^*(\phi_j)$ and $\psi_j$ can only decrease as $i$ grows so $\phi_i=\psi_i$ if $\sum_i{\phi_i}=\sum_i{\psi_i}$.
Neither is true for arbitrary $F$ in the flat lattice of the initial matroid $M_1 \oplus M_2$, but both are true ``while it matters''. More precisely, we will show that if one of these equalities fails to hold for $i$-th function, then $e_F$ is not a ray of $B(\tilde{M}_{i-1})$.
\par
After rearranging equation \ref{eq:x1x2y1y2} splits into two analogous rank identities:
\begin{equation} \label{eq:1212split}
\begin{split}
\rk_{M_1}(X_1 \cup Y_2) - \rk_{M_1}(X_1) = \rk_{N}(Y_1 \cup Y_2) - \rk_{N}(Y_1); \\
\rk_{M_2}(X_2 \cup Y_1) - \rk_{M_1}(X_2) = \rk_{N}(Y_1 \cup Y_2) - \rk_{N}(Y_2).
\end{split}
\end{equation}
First, observe that the inequality $\leqslant$ holds in both identities of equation \ref{eq:1212split} by Lemma \ref{lemma:rankunion} and due to the fact that $N$ is the restriction of $M_1$ and $M_2$, so the ranks of the subsets of $T$ are the same in $N$ as in $M_1$ or $M_2$. Moreover, if either one of these inequalities turns out to be strict, or if $\eta(F) < \rk_{M'}(F)$, then $F$ is not a flat of $M'$ (if it is, then $Y_1 = Y_2$, and equation \ref{eq:1212split} holds trivially).
\par Take the earliest step $i$ which creates the discrepancy between $\psi$ and $\pi^*(\phi)$. Namely, assume that for each $j<i$ and each ray $e_F$ of $B(\tilde{M}_{j-1})$ the equality $\psi_j(F)=\pi^*(\phi_j)(F)$ holds. Now, both sequences of numbers $\psi_j(F)$ and $\pi^*(\phi_j)(F)$ are monotonously non-increasing as $j$ grows, and
$$\sum_j \psi_j(F) = \rk_{M'}(F) - \rk_{M_1 \oplus M_2}(F) < \rk_{N}(Y_1 \cup Y_2) - \rk_{N}(Y_1)-\rk_{N}(Y_2) = \sum_j \pi^*(\phi_j)(F),$$
which means that $-1 = \psi_i(F) < \pi^*(\phi_i)(F)=0$. Our tactics is now to show that $F$ cannot actually be a flat of $\tilde{M}_{i-1}$, since there must exist a covering flat $F'$ of $F$ in $\tilde{M}_{i-2}$ which belongs to the modular cut which yields $\tilde{M}_{i-1}$, and, therefore, $e_F$ does not belong to the support of $B(\tilde{M}_{i-1})$ by Lemma \ref{modcuts}.
\par
Since $F$ is not a flat of $M'$, we can define $F'' = \cl_{M'}{F} \supsetneq F$. Now, both $F$ and $F''$ are flats of $\tilde{M}_{i-2}$ --- the former by assumption, the latter because it is even the flat of $M'$ which is at the end of the process deleting some of the flats. Take any $F' \gtrdot  F$ in $\mathcal{F}(\tilde{M}_{i-2})$ such that $F' < F''$. Since both $X_1$ and $X_2$ are flats of $M_1$ and $M_2$ respectively, we have $\rk_{M_1 \oplus M_2} F' > \rk_{M_1 \oplus M_2} F$. Therefore, for $F'$ the difference between the initial rank $\rk_{M_1 \oplus M_2}(F')$ and the final rank $\rk_{M'}(F')$ is strictly larger than for $F$, so $\sum_j \psi_j(F') < \sum_j \psi_j(F)$.
\par
Using monotonicity of $\psi$ again, we get $\psi_{i-1}(F')=-1$ (i.e., $\psi$ already was $-1$ on $F'$ on the previous step), but, at the same time, $\psi_{i-1}(F) = \pi^*(\phi_{i-1})(F)$ since $i$-th step is the first with the discrepancy, which in its turn, by monotonicity of $\pi^*{\phi}$, is not less than $\pi^*(\phi_{i})(F)=0$. For the sake of clarity these considerations are gathered in Figure \ref{fig:cube}. So, we have that $F$ is a flat under the $(i-1)$-st cut and cannot be a flat of $\tilde{M}_{i-1}$ by Lemma \ref{modcuts}. This completes the proof of Claim \ref{easy}.
\begin{figure}[h]
\includegraphics[width=\textwidth]{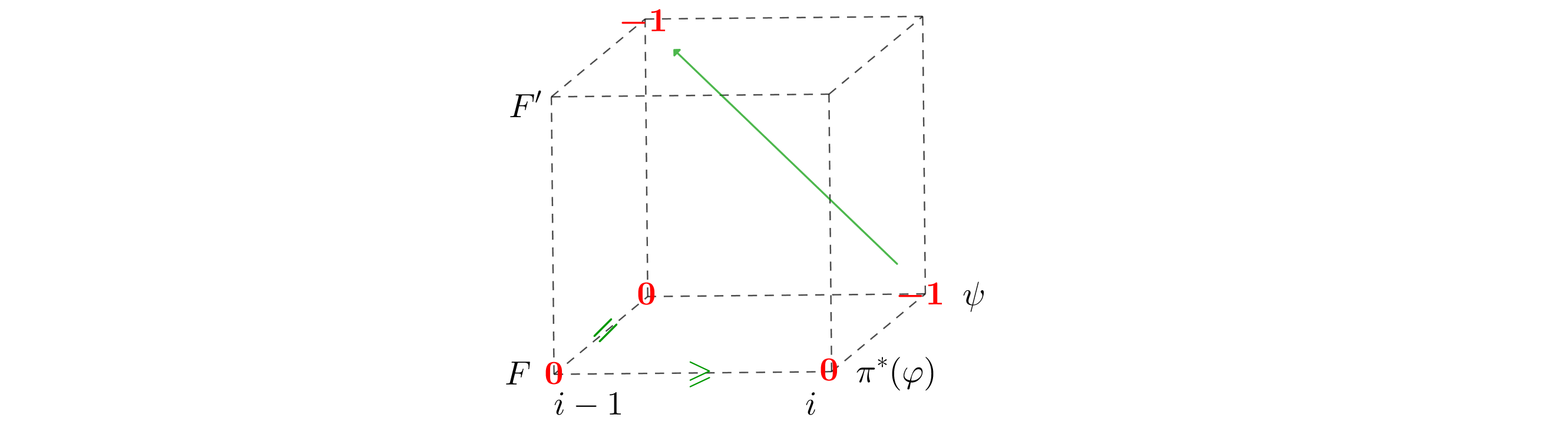}
\caption{Relations between values of $\pi^*(\phi)$ and $\psi$.}
\label{fig:cube}
\end{figure}
\par
We will now verify Claim \ref{hard}. By Claim \ref{deg1} and Theorem \ref{th:Fink} we have that $X=B(M)$ for some matroid $M$, and we want to show that $M$ is actually the proper amalgam of $M_1$ and $M_2$ along $N$ (with added parallel copies). The reason we cannot simply follow the equivalences constructed for Claim \ref{easy} in the opposite direction is that we have no analogues of Corollary 3.6 from \cite{FRau}. More precisely, we do not know what happens to the fan $\psi_{i} \cdots \psi_1 \cdot B(M_1 \oplus M_2)$ once it ceases to be a Bergman fan for the first time. Theoretically, it can emerge effective again (which actually happens, as our verification of the degree practically means that after enough truncations the fan becomes $1 \cdot \langle e_E \rangle$, with ranks of flats being not the ones prescribed by the definition of proper amalgam.
\par
By Corollary \ref{cor:tfpff}, though, $X$ is a Flag fan corresponding to some pair $(\mathcal{G},\omega_{\mathcal{G}})$. All the rays $F$ of $X$ have $Y_1=Y_2=Y$ by Lemma \ref{setdiag}, so
$$\rk_{\mathcal{G}}(F) = \rk_{M_1 \oplus M_2}(F) + \sum \pi^*(\phi_i)(F) = \rk_{M_1}(X_1) + \rk_{M_2}(X_2) - \rk_{N}(Y) = \eta(F).$$
\par Summarizing, $X = B(M)$ for some $M$ by Theorem \ref{th:Fink}, and for each flat $F$ of $M$ we have
$\rk_M(F) = \eta(F)$, which means that, by Definition \ref{def:pramalg}, $M$ is a proper amalgam of $M_1$ and $M_2$ along $N$.
\end{proof}
\section{Tropical graph correspondences} \label{sect:tropicalgraphs}
In this section we develop first steps in tropical correspondence theory. The motivation is to obtain a category of tropical fans where the tropical fibre product is actually a pullback. While this precise formulation is not achieved (and is likely not possible to achieve, see Remark \ref{remark:notpullback}), worthwhile statements are established in the process.
\subsection{Introducing notions} We begin with the following unsatisfactory observation. Consider the tropical fibre product from Example \ref{ex:tfps} (2) and a fan $B(M')$, where $M'$ is another amalgam of $M_1$ and $M_2$ with $\{3\}$ and $\{4\}$ being parallel elements. Then, neither of the supports of $B(M) = B(M_1) \times_{B(N)} B(M_2)$ and $B(M')$ is contained in the other, and it is not difficult to check (see Figure \ref{fig:nomap} for projective Bergman fans of $M$ and $M'$) that there are no tropical morphisms between them making the diagram commute:
$$\xymatrix{
B(M') \ar@/_/[ddr]_{q_1} \ar@/^/[drr]^{q_2}
\ar@{.>}[dr]|-{\nexists} \\
& B(M) \ar[d]^{p_1} \ar[r]_{p_2}
& B(M_2) \ar[d]_{\pi_2} \\
& B(M_1) \ar[r]^{\pi_1} & B(N) }$$
\begin{figure}[h]
\includegraphics[width=\textwidth]{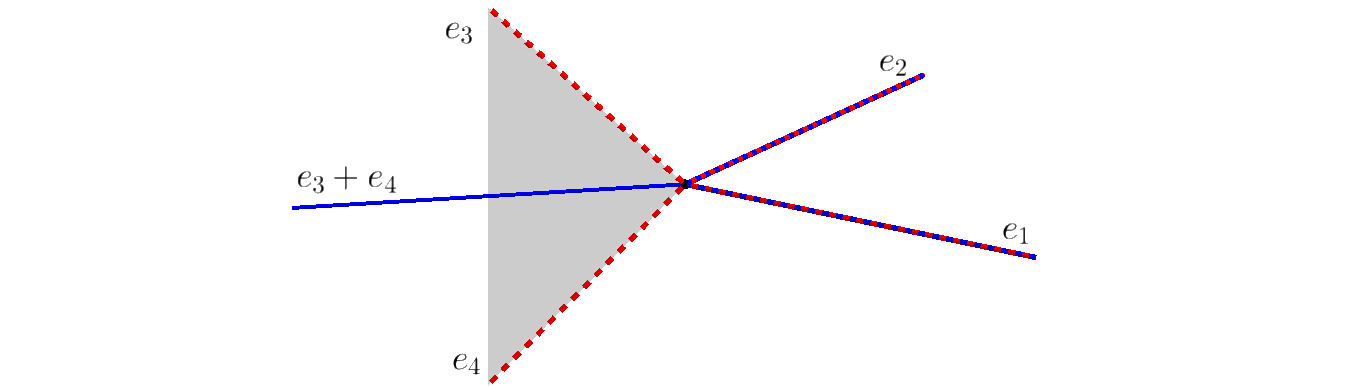}
\caption{$B(M')$ is in blue, $B(M)$ is in red.}
\label{fig:nomap}
\end{figure}
Something can be done, though:
\begin{example} \label{ex:simpletrcor}
Denote the groundset of $M'$ by $E' = \{1',2',3',4'\}$, the groundset of $M$ by $E = \{1,2,3,4\}$, and consider a Flag fan $\Gamma \subset B(M') \times B(M)$, where
$$\mathcal{G} = \{\varnothing, E \cup E', \{1',1\}, \{2',2\}, \{3',4',3\}, \{3',4',4\}, \{3',4'\}\},$$
and all edge weights are 1 except for $(\{3',4'\}, E \cup E')$ which is $-1$. It is easy to see that $\Gamma$ is actually a tropical fan and a subfan in $B(M') \times B(M)$.
\par The push-forward of $\Gamma$ onto $B(M') \times B(M_1)$ (along $\id \times p_1$) is an actual graph of tropical morphism $q_1 \colon B(M') \to B(M_1)$ (as $\{3',4',4\}, \{3',4'\}$ both map to $\{3',4'\}$ and cancel each other out, the latter having weight $-1$). The same goes for $q_2 \colon B(M') \to B(M_2)$.
\end{example}
We proceed to build on this example.
\begin{definition}
Given two simple matroids $M_1,M_2$ and their lattices of flats $\mathcal{F}_1 = \mathcal{F}(M_1), \mathcal{F}_2 = \mathcal{F}(M_2)$, the \textit{lattice morphism} is a map $f \colon \mathcal{F}_1 \to \mathcal{F}_2$ such that if $F >_{\mathcal{F}_1} F'$, then $f(F) \geqslant_{\mathcal{F}_2} f(F')$.
\par We call a lattice morphism $f$ a \textit{weak lattice map} if $f(F)\neq\varnothing$ for $F \neq\varnothing$ and $\rk_{M_1}F \geqslant \rk_{M_2}(f(F))$.
\par We call a weak lattice map \textit{covering} if whenever $F \gtrdot _{\mathcal{F}_1} F'$, either $f(F) \gtrdot _{\mathcal{F}_2} f(F')$ or $f(F)=f(F')$.
\end{definition}
Note that covering lattice maps are automatically weak provided that $f(F)\neq\varnothing$ for $F \neq\varnothing$.
\par Simple matroids with weak lattice maps form a category $\mathbf{SMatrWL}$, and covering maps form a wide subcategory $\mathbf{SMatrWL_{\gtrdot }}$. The category of simple matroids with weak maps of matroids (Definition \ref{def:strongweakmaps}) is denoted by $\mathbf{SMatrW}$. There is a forgetful functor $\mathrm{Pt} \colon \mathbf{SMatrWL} \to \mathbf{SMatrW}$: it maps simple matroids to themselves, and, given $f \colon \mathcal{F}_1 \to \mathcal{F}_2$, the map $\mathrm{Pt}(f) \colon E_1 \to E_2$ is defined via $\mathrm{Pt}(f)(x) = f(\{x\})$. The definition is correct since $\rk_{M_2} f(\{x\}) = 1$, and all rank-1 flats of $M_2$ are one-element subsets since $M_2$ is simple.
\par The map $\mathrm{Pt}(f) \colon E_1 \to E_2$ is actually a weak map: if $X \subset E_1$, then $$\rk_{M_2} \mathrm{Pt}(f)(X) = \rk_{M_2}{\left(\cl_{M_2}{\bigcup_{x \in X} f(\{x\})}\right)} = \rk_{M_2}{\bigvee_{x \in X} f(\{x\})} \leqslant \rk_{M_2}{f \left(\bigvee_{x \in X} \{x\}\right)} \leqslant \rk_{M_1} X.$$
In the examples of this section matroids shown are sometimes not simple for the sake of clarity, but what is assumed every time is their simplifications, where each flat of rank $1$ is replaced by a single element.
\par Not only functor $\mathrm{Pt}$ is not faithful, it also does not have the right inverse, as shown by the following
\begin{example} \label{ex:three}
On Figure \ref{fig:threemaps} there are three different weak lattice maps from the matroid on the left side to the matroid on the right side. The only difference between them is the image of the flat $\{3,4\}$ --- in the top map it is $\{1,3,4\}$, in the bottom map it is $\{2,3,4\}$, and in the central map it is $\{3,4\}$ again. Note that the central map is not a covering lattice map ($\{1,2,3,4\}$ covers $\{3,4\}$ in the left matroid, but not in the right one), while the top and the bottom are compositions of covering lattice maps and, therefore, covering lattice maps themselves. For each of the four top and bottom arrows there is only one lattice map corresponding to the identity map on the groundset. Thus, no matter which lattice map we consider to be the image of the central identity map on the groundset, the map will not be functorial. 
\begin{figure}[h]
\includegraphics[width=\textwidth]{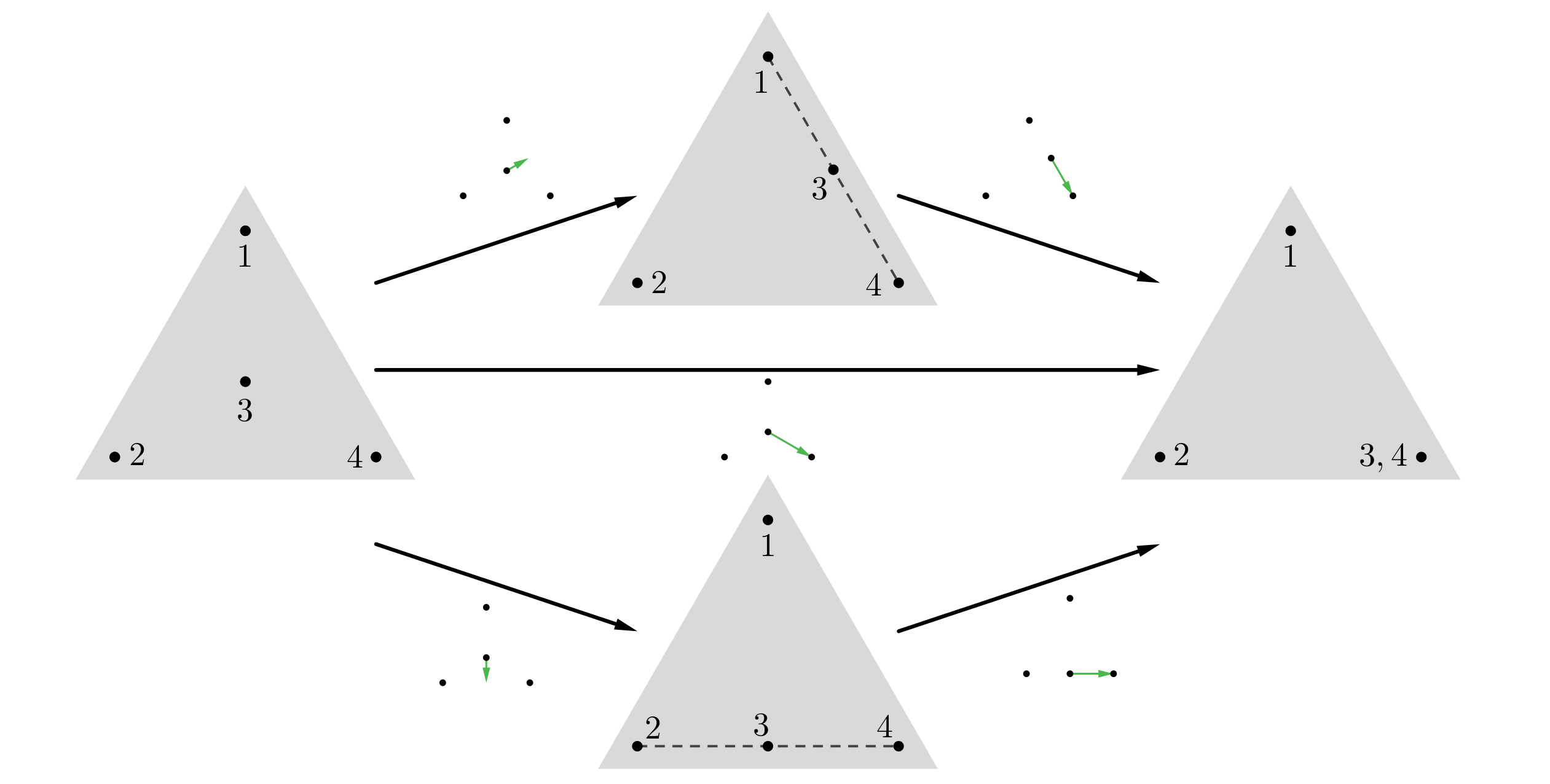}
\caption{Three weak lattice maps with the same weak map of groundset.}
\label{fig:threemaps}
\end{figure}
\end{example}
\par Next, we define tropical correspondences between Bergman fans. The whole construction is almost verbatim to the correspondences between, say, smooth projective varieties over $\cc$, with the (already noted in tropical intersection theory) exception that those are defined on classes of cycles modulo rational equivalence, while groups of tropical cycles are not quotients.
\begin{definition} \label{def:tropcor}
A \textit{tropical correspondence} $\Gamma$ between Bergman fans $B_X = B(M_1)$ and $B_Y = B(M_2)$ denoted by $\Gamma \colon B_X \vdash B_Y$ is a tropical subcycle in $B_X \times B_Y$. The \textit{identity} correspondence is the diagonal $\Delta_X \subset B_X \times B_X$. The composition of correspondences $\Gamma_{XY} \colon B_X \vdash B_Y$ and $\Gamma_{YZ} \colon B_Y \vdash B_Z$ is given by the following formula:
$$\Gamma_{YZ} \circ \Gamma_{XY} \defeq (\pi_{XZ})_*(\pi_{XY}^*(\Gamma_{XY}) \cdot \pi_{YZ}^*(\Gamma_{YZ})) \subset B_X \times B_Z,$$
where $\pi_{XY} \colon B_X \times B_Y \times B_Z \to B_X \times B_Y$ is a projection, and similarly $\pi_{XZ},\pi_{YZ}$.
\end{definition}
\begin{lemma}
Bergman fans of simple matroids with tropical correspondences form a category (that we denote with $\mathbf{SMatrCor}$):
\begin{itemize}
\item $\Delta_X \colon B_X \vdash B_X$ is a unit: for any $\Gamma \colon B_X \vdash B_Y \; \Gamma \circ \Delta_X = \Gamma$, and for any $\Gamma \colon B_Y \vdash B_X \; \Delta_X \circ \Gamma = \Gamma$;
\item Composition is associative: for any $\Gamma_{12} \colon X_1 \vdash X_2, \Gamma_{23} \colon X_2 \vdash X_3, \Gamma_{34} \colon X_3 \vdash X_4$, we have
$$\Gamma_{34} \circ (\Gamma_{23} \circ \Gamma_{12}) = (\Gamma_{34} \circ \Gamma_{23}) \circ \Gamma_{12}. $$
\end{itemize}
\end{lemma}
\begin{proof}
This is standard: see, for example, Proposition 16.1.1 of \cite{Fulton}. The only non-trivial ingredient needed is the projection formula --- Theorem \ref{th:projformula}. Note that Proposition 1.7 of \cite{Fulton}, that is used as a justification of one of the equalities in the proof, is actually applied to show that
$$(\pi^{1234}_{123})_* \circ (\pi^{1234}_{124})^* = (\pi^{123}_{12})^* \circ (\pi^{124}_{12})_*,$$
where $\pi^{i_1 \ldots i_k}_{j_1 \ldots j_l}$ denotes the projection $X_{i_1} \times \ldots \times X_{i_k} \to X_{j_1} \times \ldots \times X_{j_l}$. This particular case is obviously true in $\mathbf{SMatrCor}$.
\end{proof}
Category $\mathbf{SMatrCor}$ has too many morphisms to be useful for us here. In the category of correspondences between smooth projective varieties there is a subcategory of graphs of morphisms of varieties. We want to imitate this subcategory, but our starting point (see Example \ref{ex:simpletrcor} and discussion above it) is that actual graphs of tropical morphisms are not sufficient. Therefore, we establish a larger subcategory which we expect to behave like graphs.
\begin{definition} \label{def:graphcor}
Given matroids $M_1, M_2$ and a weak lattice map $f \colon \mathcal{F}_1 \to \mathcal{F}_2$, we define a \textit{graph rank function} $\gamma(f) \colon \mathcal{F}(M_1 \oplus M_2) \to \zz_{\leqslant 0}$ by the rule:
$$\gamma(f)((F_1,F_2)) = \rk_{M_2}(f(F_1) \cup F_2) - \rk_{M_1} F_1 - \rk_{M_2} F_2.$$
Similar to Lemma \ref{modcuts}, a graph rank function $\gamma(f)$ is decomposed as a sum of $\gamma_i(f) \colon \mathcal{F}(M_1 \oplus M_2) \to \{-1, 0\}, 1 \leqslant i \leqslant \rk{M_1}$ in such a way that value on each $F=(F_1,F_2)$ does not increase as $i$ grows. More precisely,
$$\gamma_i(f)((F_1,F_2)) = \begin{cases}
0, \mbox{ if $\rk_{M_2}(f(F_1) \cup F_2) + \rk{M_1}-i \geqslant \rk_{M_1} F_1 + \rk_{M_2} F_2$}  \\
-1, \mbox{ otherwise.}
\end{cases}$$
A \textit{graph correspondence} $\Gamma_f$ is given by
$$\prod_{i=1}^{\rk M_1} \gamma_i(f) \cdot (B(M_1) \times B(M_2)).$$
\end{definition}
\begin{example} \label{ex:grcors} Here are a few instances of Definition \ref{def:graphcor}:
\begin{enumerate}
\item For the identity lattice map $ \iota \colon \mathcal{F}(N) \to \mathcal{F}(N)$ we get $\gamma_i(\iota) = \phi_i$ from Definition \ref{diagconstr}, and thus $\Gamma_{\iota} = \Delta \subset B(N) \times B(N)$.
\item More generally, if $f \colon E_1 \to E_2$ is a strong map between groundsets of matroids $M_1$ and $M_2$, as in Definition \ref{def:strongweakmaps}, and $f^* \colon \rr^{E_2} \to \rr^{E_1}$ is the dual map, then the set-theoretical graph
$$\{(v, f^*(v)), v \in B(M_2)\} \subset B(M_2) \times B(M_1)$$
is a Bergman fan $B(M_f)$ of matroid $M_f$ on the groundset $E_1 \sqcup E_2$. Matroid $M_f$ is isomorphic to $M_2$, with $x \in E_1$ being parallel element to $f(x) \in E_2$. By Lemma \ref{modcuts}, since $|B(M_f)| \subset |B(M_1 \oplus M_2)|$, it is cut by the rank functions, and it is easy to see that they coincide with $\gamma_i(f_{\mathcal{F}})$, where $f_{\mathcal{F}} \colon \mathcal{F}(M_1) \to \mathcal{F}(M_2)$ is the lattice map induced by $f$. Thus, $\Gamma_{f_{\mathcal{F}}}$ equals set-theoretic graph of $f^*$. 
\item If $M$ and $M'$ are as in Example \ref{ex:simpletrcor}, and $f \colon \mathcal{F}(M) \to \mathcal{F}(M')$ is given by $f(F) = \cl_{M'}(F)$, then the Flag fan $\Gamma \subset B(M') \times B(M)$ that we constructed in the example equals $\Gamma_f$.
\end{enumerate}
\end{example}
We are going to see that Example \ref{ex:grcors} (3) is not a coincidence.
\begin{lemma} \label{lemma:covcorisff}
The graph correspondence $\Gamma_f$ of the covering lattice map $f$ is a Flag fan.
\end{lemma}
\begin{proof}
We need to verify that each $\gamma_i(f)$ is monotonous on flats of $M_1 \oplus M_2$ and, therefore, on each of the Flag fans $\prod_{i=1}^{j} \gamma_i(f) \cdot (B(M_1) \times B(M_2))$ on the way (their posets are the subposets of the initial lattice of flats). Then, the claim follows from Lemma \ref{lemma:ff}.
\par Since $\gamma_i(f)$ are monotonously non-increasing as $i$ grows, and $\sum_i \gamma_i(f) = \gamma(f)$, we want to show that, if $F' \gtrdot F$, then $\gamma(f)(F') \leqslant \gamma(f)(F)$. The covering relations $F' \gtrdot_{M_1 \oplus M_2}  F$ can be of two kinds: $F' = (F'_1, F_2) \gtrdot (F_1, F_2) = F$, where $F'_1 \gtrdot_{M_1} F_1$, and $F' = (F_1, F'_2) \gtrdot  (F_1, F_2) = F$, where $F'_2 \gtrdot_{M_2} F_2$. As $\gamma(f)((F_1,F_2)) = \rk_{M_2}(f(F_1) \cup F_2) - \rk_{M_1} F_1 - \rk_{M_2} F_2$, and exactly one of the subtrahends increases by $1$ when replacing $F$ with $F'$, we want to show that
$$\rk_{M_2}(f(F'_1) \cup F_2) \leqslant \rk_{M_2}(f(F_1) \cup F_2)+1 \; \mbox{  and  } \; \rk_{M_2}(f(F_1) \cup F'_2) \leqslant \rk_{M_2}(f(F_1) \cup F_2)+1.$$
Since $f$ is a covering lattice map, $\rk_{M_2}{f(F'_1)} \leqslant \rk_{M_2}{f(F_1)}+1$, so both inequalities boil down to the fact that in any matroid $M$ with groundset $E$ and for any $A \subset E$
$$F' \gtrdot_M F \Rightarrow \rk_M(F' \cup A) - \rk_M(F \cup A) \leqslant 1,$$
which is simply Lemma \ref{lemma:rankunion}.
\end{proof}
\subsection{Graph functor} In this subsection we prove the second main result of the paper --- that graph correspondence from Definition \ref{def:graphcor} provides a functor from $\mathbf{SMatrWL_{\gtrdot}}$ to $\mathbf{SMatrCor}$ (Theorem \ref{th:functor}). In order to do that, we first describe the edge weights of graph correspondence $\Gamma_f$ combinatorially, in terms of $f$, in Theorem \ref{th:grcorstr}, using that $\Gamma_f$ is a Flag fan due to Lemma \ref{lemma:covcorisff}. It requires some preparation. Recall the following
\begin{definition}
A Möbius function $\mu$ on poset $\mathcal{G}$ is defined for pairs $A \leqslant_{\mathcal{G}} B$ recursively via
$$\mu(A,A)=1, \; \; \mu(A,B) = - \sum_{A \leqslant C < B} \mu(A,C).$$
One sees that function $\mu$ is designed in such a way that $\sum_{A \leqslant C \leqslant B} \mu(A,C) = \delta(A,B)$, where $\delta$ is a delta-function, equal to $1$ if $A=B$ and $0$ otherwise.
\end{definition}
We are going to use the following simple property of $\mu$, which can be found, for example, in \cite{Stanley} (Ch. 3, Exercise 88).
\begin{lemma} \label{lemma:muproperty}
For a poset $\mathcal{P}$ with the minimal element (denoted by $\varnothing$ for convenience)
$$\sum_{a,b \in \mathcal{P}}{\mu(a,b)}=1.$$
\end{lemma}
\begin{proof}
We have
$$\sum_{a,b \in \mathcal{P}}{\mu(a,b)}= \sum_{b \in \mathcal{P}}{\sum_{a \leqslant b}{\mu(a,b)}}=\sum_{b \in \mathcal{P}}{\sum_{\varnothing \leqslant a \leqslant b}{\mu(a,b)}} = \mu(\varnothing,\varnothing) = 1,$$
where the penultimate equality follows from $\sum_{A \leqslant C \leqslant B} \mu(A,C) = \delta(A,B)$.
\end{proof}
\begin{remark} \label{remark:inclexcl}
We are going to use Lemma \ref{lemma:muproperty} several times, so, to avoid excessive combinatorial abstraction, we define $\eta(x) = \eta_{\mathcal{P}}(x) = \sum_{b \geqslant_{\mathcal{P}} x} \mu(x,b)$. Lemma \ref{lemma:muproperty} can be rewritten, then, as $\sum_{x \in \mathcal{P}}{\eta(x)}=1$ whenever $\mathcal{P}$ has the minimal element. Observe that, if $\mathcal{P}$ is a meet-semilattice, $\eta(x)$ is the coefficient of $x$ in the inclusion-exculsion formula written for the maximal elements $A_1, \ldots, A_n$ of the $\mathcal{P}$. Indeed,
$$\eta(x) = \sum_{b \geqslant x} \mu(x,b) = \sum_{1 \leqslant i_1 < \ldots < i_k \leqslant n}{(-1)^k \sum_{\bigcap_{j=1}^k A_{i_j} \geqslant b \geqslant x}{\mu(x,b)}} =  \sum_{1 \leqslant i_1 < \ldots < i_k \leqslant n}{(-1)^k \cdot \delta \left(\bigcap_{j=1}^k A_{i_j},x \right)}.$$
\end{remark}
Before formulating Theorem \ref{th:grcorstr} rigorously, let us give its informal explanation. Basically, it claims that the complete flags of the graph correspondence $\Gamma_f$ are of the form
$$\varnothing \lessdot (X_1,Y_1) \lessdot \ldots \lessdot (E(M_1), E(M_2)),$$
where $\varnothing \lessdot Y_1 \ldots \lessdot E(M_2)$ is a complete flag in $M_2$, images $f(X_i)$ lie under $Y_i$ in $\mathcal{F}(M_2)$ (but not necessarily equal to them), and the weights are balanced in such a way that the whole thing is a tropical fan and its push-forward onto $B(M_2)$ is $B(M_2)$. This way resembles inclusion-exclusion principle (see Remark \ref{remark:inclexcl}).
\begin{theorem} \label{th:grcorstr}
Let $\Gamma_f \colon B(M_1) \vdash B(M_2)$ be a graph correspondence of the covering lattice map $f \colon \mathcal{F}(M_1) \to \mathcal{F}(M_2)$. Let $(\mathcal{G},\omega_{\mathcal{G}})$ be the pair for which $\Gamma_f$ is a Flag fan. Then, all the edges of $\mathcal{G}$ are as follows:
$$(X_1,Y_1) \lessdot_{\mathcal{G}} (X_2,Y_2), \mbox{ where } \; Y_1 \lessdot_{M_2} Y_2, \; f(X_1) \leqslant_{M_2} Y_1 \mbox{ and } X_1 \leqslant_{M_1} X_2.$$
Denote by $\mathcal{F}(X_1,X_2,Y_1)$ the subposet of $\mathcal{F}(M_1)$ consisting of $F$ such that
$$X_1 \leqslant_{M_1} F \leqslant_{M_1} X_2 \; \mbox{ and } \; f(F) \leqslant_{M_2} Y_1.$$
Then,
$$\omega_{\mathcal{G}}((X_1,Y_1),(X_2,Y_2))=\eta_{\mathcal{F}(X_1,X_2,Y_1)}(X_1).$$
\end{theorem}
\begin{proof}
Denote the Flag fan $\gamma_{i}(f) \cdots \gamma_1(f) \cdot (B(M_1) \times B(M_2))$ by $\Gamma_{f,i} = \Gamma_{i}$, and the corresponding poset and weights pair by $(\mathcal{G}_{i},\omega_{\mathcal{G}_{i}})$. It turns out to be convenient to digress for a moment from considering edges and to focus on the vertices of $\mathcal{G}$. Notice that, if we show that only pairs $(X,Y) \in \mathcal{F}(M_1 \oplus M_2)$ such that $f(X) \leqslant_{M_2} Y$ can belong to $\mathcal{G}$, then the first claim follows. Indeed, the inequality $X_1 \leqslant_{M_1} X_2$ follows from the fact that $\mathcal{G}$ is the subposet of the initial lattice $\mathcal{F}(M_1 \oplus M_2)$, and $Y_1 \lessdot_{M_2} Y_2$ is obtained from the comparison of ranks in $\mathcal{G}$. More precisely, observe that $\gamma(f)((X,Y)) = -\rk_{M_1}{X}$ since $Y = Y \vee f(X)$, and their ranks in Definition \ref{def:graphcor} cancel each other out. Thus, 
$$\rk_{\mathcal{G}}{(X,Y)} = \rk_{M_1 \oplus M_2}{(X,Y)} + \gamma(f)((X,Y)) = \rk_{M_2}{Y},$$
which means that if $Y_2 >_{M_2} Y_1$ but $\rk_{M_2}{Y_2} - \rk_{M_2}{Y_1} > 1$, then there can be no edge in $\mathcal{G}$ between $(X_1,Y_1)$ and $(X_2,Y_2)$. But then there can be no edges between $(X_1,Y)$ and $(X_2,Y)$ in $\mathcal{G}$ as well, because any increasing chain must reach $(E(M_1),E(M_2))$ from $\varnothing$ in $\rk{\mathcal{G}}=\rk{M_2}$ steps, with the rank in $M_2$ not allowed to jump, so it must grow steadily by 1 on each edge of $\mathcal{G}$.
\par Assume the contrary: there exists $(X,Y) \in \mathcal{G}$ with $Y \neq f(X) \vee Y$. Take any maximal vertex with this property. Observe that, by definition of $\gamma(f)$, $\rk_{\mathcal{G}}{(X,Y)} = \rk_{\mathcal{G}}{(X,f(X) \vee Y)}$, meaning that the pair $(X,f(X) \vee Y)$ has its rank decreased by $1$ with each next $\mathcal{G}_i$ until minimal $i$ such that $\rk_{\mathcal{G}_i}{(X,Y)} = \rk_{\mathcal{G}_i}{(X,f(X) \vee Y)}$, and then they continue the descent together. Consider the last step before $(X,f(X) \vee Y)$ reaches $(X,Y)$, namely, maximal $i$ that $\gamma_i(f)((X,Y))=0$. We will prove that the localization $\gamma_i(f)^{e_{(X,Y)}}$ is linear on the star of $\mathrm{Star}_{\Gamma_{i-1}}e_{(X,Y)}$, so, by Lemma \ref{lemma:stardivisor}, $(X,Y)$ does not belong to the next poset $\mathcal{G}_i$ and, consequently, to $\mathcal{G}$.
\par We show that $\gamma_i(f)^{e_{(X,Y)}} = - \chi_y$, where $y \in (f(X) \vee Y) \setminus Y \subset E(M_2)$ is any element. Since the characteristic function $\chi_y$ is linear on the $\mathrm{Star}_{\Gamma_{i-1}}e_{(X,Y)}$, the claim follows. Note that we are not interested in the whole fan structure of the star, only that the rays are the elements $(X',Y')$ comparable with $(X,Y)$. Those of them which are below $(X,Y)$ have $\gamma_i(f)^{e_{(X,Y)}}$ equal to $0$. For $(X',Y') \geqslant_{\mathcal{G}_{i-1}}(X,Y)$ we need to verify the equivalence:
$$y \in Y' \Leftrightarrow \gamma_i(f)((X',Y')) = -1.$$
If $y \notin Y'$, then, in particular, $f(X') \vee Y' \supset f(X) \vee Y \ni y \notin Y'$. Since $\gamma_i(f)((X',Y'))$ is already $-1$, the next $\gamma_j(f)((X',Y'))$ are going to be $-1$ for all $j \geqslant i$, so the vertex $(X',Y')$ never disappears and thus belongs to $\mathcal{G}$, which contradicts $(X,Y)$ being a maximal element of $\mathcal{G}$ with $Y \neq f(X) \vee Y$. This proves $\Leftarrow$.
\par If $y \in Y'$, then 
\begin{equation} \label{eq:strictgamma}
\begin{split}
\gamma(f)((X',Y')) \leqslant \gamma(f)((X,Y')) = \rk_{M_2}(f(X) \cup Y') - \rk_{M_1}{X} - \rk_{M_2}{Y'} < \\ <
\rk_{M_2}(f(X) \cup Y) - \rk_{M_1}{X} - \rk_{M_2}{Y} = \gamma(f)((X,Y)),
\end{split}
\end{equation}
which implies $\gamma_i(f)((X',Y')) = -1$ since $\gamma_{i+1}(f)((X,Y)) = -1$. The strict inequality in equation \ref{eq:strictgamma} follows from Lemma \ref{lemma:rankunion} and the fact that $y \in f(X) \vee Y$. This proves $\Rightarrow$.
\par It remains to calculate $\omega_{\mathcal{G}}((X_1,Y_1),(X_2,Y_2))$ in the case where both vertices belong to $\mathcal{G}$. We will perform it by induction on the number of step $i$ when the edge $((X_1,Y_1),(X_2,Y_2))$ appears in $\mathcal{G}_i$ (step 2 of Corollary \ref{cor:alg}). This is made possible by the fact that, as we are going to see, all the edge weights on which $\omega_{\mathcal{G}}((X_1,Y_1),(X_2,Y_2))$ depends are determined for earlier $i$.
\par Poset $\mathcal{F}(X_1,X_2,Y_1)$ does not always have the top element (which is precisely the reason for the weights to get complicated), but it always has the bottom element $X_1$. Therefore, in the base of induction $\mathcal{F}(X_1,X_2,Y_1)$ consists of a single element $X_1$. Indeed, if there is another $X' \in \mathcal{F}(X_1,X_2,Y_1)$, the edge $((X',Y_1),(X_2,Y_2))$ would appear before, since
$$\gamma(f)((X',Y_1)) = -\rk_{M_1}{X'} < -\rk_{M_1}{X_1} = \gamma(f)((X_1,Y_1)).$$
We aim to show that $\omega_{\mathcal{G}}((X_1,Y_1),(X_2,Y_2))=1$.
\par The edge between vertices $(X_1,Y_1)$ and $(X_2,Y_2)$ appears when taking the Weil divisor of $\gamma_i(f)$ for such $i$ that $\gamma_i(f)(((X_1,Y_1))) = 0$, $\gamma_i(f)(((X_2,Y_2))) = -1$ and $\gamma_{i+1}(f)(((X_1,Y_1))) = -1$. As we have seen in Corollary \ref{cor:alg}, the weight on the newly-created edge is equal to
$$\deg{(\gamma_i(f)|_{(\Gamma_{i-1})_{(X_1,Y_1),(X_2,Y_2)}} \cdot (\Gamma_{i-1})_{(X_1,Y_1),(X_2,Y_2)})},$$
which, after expanding according to Definition \ref{def:Weil}, yields
\begin{equation} \label{eq:graphweight}
\begin{split}
-\gamma_i(f) \left(\sum_{(X_1,Y_1) \lessdot_{\mathcal{G}_{i-1}} F' \lessdot_{\mathcal{G}_{i-1}} (X_2,Y_2)} \omega_{\mathcal{G}_{i-1}}((X_1,Y_1),F')\omega_{\mathcal{G}_{i-1}}(F',(X_2,Y_2)) \cdot e_{F'} \right) + \\ +
\sum_{(X_1,Y_1) \lessdot_{\mathcal{G}_{i-1}} F' \lessdot_{\mathcal{G}_{i-1}} (X_2,Y_2)}
\omega_{\mathcal{G}_{i-1}}((X_1,Y_1),F')\omega_{\mathcal{G}_{i-1}}(F',(X_2,Y_2)) \cdot \gamma_i(f)(F').
\end{split}
\end{equation}
Due to the balancing condition written for the hyperface of $\Gamma_{i-1}$ generated by any chain containing $(X_1,Y_1)$ and $(X_2,Y_2)$ with missing $F'$, the sum in the large brackets is just some $c \cdot e_{(X_2,Y_2)}$. Then, since $\gamma_i(f)((X_2,Y_2)) = -1$, the first summand is just $c$. We want to show that $\gamma_i(f)(F')=0$ for all $F'$ in the sum, so that the second term is $0$, and that $c=1$.
\par Recall that $Y_2 \gtrdot_{M_2} Y_1$, thus every $F'$ has the form $(X',Y_1)$ or $(X',Y_2)$. In the former case, $X_2 \geqslant_{M_1} X' >_{M_1} X_1$, but $\mathcal{F}(X_1,X_2,Y_1)$ consists of $X_1$ only, therefore, $Y_1 \neq f(X') \vee Y_1$. Thus, if $\gamma_i(f)((X',Y_1))=-1$, we get a contradiction (this vertex should not exist already in $\mathcal{G}_{i-1}$). In the latter case, since $f(X') \leqslant_{M_2} f(X_2) \leqslant_{M_2} Y_2$, values of $\gamma(f)$ on both $(X',Y_2)$ and $(X_2,Y_2)$ are equal to the ranks $\rk_{M_1}{X'},\rk_{M_1}{X_2}$, respectively, and $\gamma_i(f)$ cannot coincide on them unless they are of the same rank in $\mathcal{G}_{i-1}$, in which case $(X_2,Y_2) \gtrdot_{\mathcal{G}_{i-1}} (X',Y_2)$ could not occur.
\par To see that $c=1$, focus on those $F'$ of the form $(X',Y_2)$ (in other words, count only the coordinate of any element $y \in Y_2 \setminus Y_1$ in the vector $c \cdot e_{(X_2,Y_2)}$). As we have just seen, if $X' > X_1$, then $(X_2,Y_2)$ cannot cover $(X',Y_2)$ in $\mathcal{G}_{i-1}$, so there is only one such $F' = (X_1,Y_2)$. Now, $\omega_{\mathcal{G}_{i-1}}((X_1,Y_1),(X_1,Y_2))=1$, because this is the original covering edge from $\mathcal{F}(M_1 \oplus M_2)$, where all the weights are $1$, so it remains to prove that $\omega_{\mathcal{G}_{i-1}}((X_1,Y_2),(X_2,Y_2))=1$ as well. This is easy, though: $\gamma(f)$ coincides with $-\rk_{M_1}$ on the whole segment $[(X_1,Y_2),(X_2,Y_2)]$ of $\mathcal{F}(M_1 \oplus M_2)$, and thus all newly-created edges always have weight $1$ there. Indeed, for each $i$ value of $\gamma_i(f)(X'',Y_2)$ equals $-1$ if $X''$ is above certain rank in $M_1$, so the value of each $\gamma_i(f)$ on the level of missing flats is $0$, thus they never impact the weight of the new edge. It follows that there is only one $F'$ containing $Y_2$, and its weight is 1, so $c=1$ as claimed. This completes the base of induction.
\par Let us now find $\omega_{\mathcal{G}}((X_1,Y_1),(X_2,Y_2))$ in the case of arbitrary $\mathcal{F}(X_1,X_2,Y_1)$. We begin as previously, fixing $i$ such that this edge appears when taking the Weil divisor of $\gamma_i(f)$, and writing equation \ref{eq:graphweight} for its weight. This time, though, we already know that
$$\omega_{\mathcal{G}_{i-1}}((X',Y_1),(X_2,Y_2))= \eta_{\mathcal{F}(X',X_2,Y_1)}(X')$$
for $X' >_{M_1} X_1$ that belong to $\mathcal{F}(X_1,X_2,Y_1)$ by the induction assumption. We have also seen that $\omega_{\mathcal{G}_{i-1}}((X_1,Y_1),(X',Y_1))=1$ (as in the previous paragraph, $\gamma(f)$ coincides with $-\rk_{M_1}$ on the whole segment $[(X_1,Y_1),(X',Y_1)]$). Thus, the second term of equation \ref{eq:graphweight} becomes
$$-\sum_{X \neq X' \in \mathcal{F}(X_1,X_2,Y_1)} \left(\eta_{\mathcal{F}(X',X_2,Y_1)}(X') \right).$$
The first term still equals $1$, since $c=1$ by exactly the same argument as before --- $(X_1,Y_2)$ is the only vertex between $(X_1,Y_1)$ and $(X_2,Y_2)$ containing elements of $Y_2 \setminus Y_1$, and both edges connecting it to $(X_1,Y_1)$ and $(X_2,Y_2)$ have weight $1$ in $\omega_{\mathcal{G}_{i-1}}$. Thus, we need
$$\eta_{\mathcal{F}(X_1,X_2,Y_1)}(X_1) = 1 -\sum_{X \neq X' \in \mathcal{F}(X_1,X_2,Y_1)} \left(
\eta_{\mathcal{F}(X',X_2,Y_1)}(X') \right),$$
which rearranges to
$$\sum_{X' \in \mathcal{F}(X_1,X_2,Y_1)} \eta_{\mathcal{F}(X',X_2,Y_1)}(X') = 1,$$
and this is simply Lemma \ref{lemma:muproperty}.
\end{proof}
We can now use Theorem \ref{th:grcorstr} to obtain
\begin{theorem} \label{th:functor}
In the notation of Definition \ref{def:graphcor}, a map between Hom-sets of $\mathbf{SMatrWL_{\gtrdot}}$ and $\mathbf{SMatrCor}$ given by $f \mapsto \Gamma_f$ is a functor: if $f \colon \mathcal{F}(M_1) \to \mathcal{F}(M_2)$ and $g \colon \mathcal{F}(M_2) \to \mathcal{F}(M_3)$ are covering lattice maps between simple matroids, then
$$\Gamma_{g \circ f} = \Gamma_g \circ \Gamma_f.$$
\end{theorem}
\begin{proof} Despite seemingly cumbersome construction of $\Gamma_g \circ \Gamma_f$ and the fact that in general we have no idea whether the push-forward of the Flag fan is again a Flag fan, or, even if it is, what are its $(\mathcal{G},\omega_{\mathcal{G}})$ --- despite all that, we already possess almost all the necessary tools to verify the claim.
\par For simplicity, denote $B(M_1)=B_X,B(M_2)=B_Y,B(M_3)=B_Z$. Thus,
$$\Gamma_f \subset B_X \times B_Y, \; \Gamma_g \subset B_Y \times B_Z, \; \Gamma_{g \circ f} \subset B_X \times B_Z,$$
and we need to show that
$$(\pi_{XZ})_* \left[(\Gamma_f \times B_Z)\cdot(B_X \times \Gamma_g)\right] = \Gamma_{g \circ f}.$$
Using the fact that, for the subcycles of Bergman fans which are cut by Weil divisors of rational functions $\phi_i$, the intersections via diagonal construction and via $\phi_i$'s coincide (\cite{FRau}, Theorem 4.5(6)), we will check the following:
$$(\pi_{XZ})_* \left( \prod \pi_{XY}^* (\gamma(f)) \cdot (B_X \times \Gamma_g) \right) = \prod \gamma(g \circ f) \cdot (B_X \times B_Z).$$
\par To control edge weights of $\prod \pi_{XY}^* (\gamma(f)) \cdot (B_X \times \Gamma_g)$, where $B_X \times \Gamma_g$ is not necessarily a Bergman fan anymore, we will need an almost verbatim generalization of Theorem \ref{th:grcorstr} which we formulate as a standalone lemma.
\begin{lemma} \label{lemma:XYZ}
In the notation of Theorem \ref{th:functor}, if $(\mathcal{H},\omega_{\mathcal{H}})$ is the pair corresponding to the Flag fan $\prod \pi_{XY}^* (\gamma_i(f)) \cdot (B_X \times \Gamma_g)$, then all the edges of $\mathcal{H}$ are as follows:
$$(X_1,Y_1,Z_1) \lessdot_{\mathcal{H}} (X_2,Y_2,Z_2), \mbox{ where } \; (Y_1,Z_1) \lessdot_{\mathcal{G}_{YZ}} (Y_2,Z_2), \; f(X_1) \leqslant_{M_2} Y_1 \mbox{ and } X_1 \leqslant_{M_1} X_2.$$
The weight of the edge is given by the formula
$$\omega_{\mathcal{H}}((X_1,Y_1,Z_1),(X_2,Y_2,Z_2)) = \omega_{\mathcal{G}_{YZ}}((Y_1,Z_1),(Y_2,Z_2)) \cdot \eta_{\mathcal{F}(X_1,X_2,Y_1)}(X_1).$$
\end{lemma}
\begin{proof}
By Lemma \ref{lemma:ffproduct} the edges of $B_X \times \Gamma_g$ before cutting Weil divisors are of two types. The first type is $((X_1,Y,Z),(X_2,Y,Z))$, where $X_1 \lessdot_{M_1} X_2$ and $(Y,Z) \in \mathcal{G}_{YZ}$. The weight of such an edge is $1$, as in $\mathcal{F}(M_1)$. The second type is $((X,Y_1,Z_1),(X,Y_2,Z_2))$, where $(Y_1,Z_1) \lessdot_{\mathcal{G}_{YZ}} (Y_2,Z_2)$. The weight of such an edge is $\eta_{\mathcal{F}(Y_1,Y_2,Z_1)}(Y_1)$.
\par The claim is then verified analogously to Theorem \ref{th:grcorstr}. First, we show that only vertices with $f(X_1) \leqslant_{M_2} Y_1$ survive in $\mathcal{H}$. It follows that for any edge $((X_1,Y_1,Z_1),(X_2,Y_2,Z_2))$ of $\mathcal{H}$ the pair $(Y_2,Z_2)$ must cover $(Y_1,Z_1)$ in $\mathcal{G}_{YZ}$. To show this, as before, assume the contrary, and take any maximal surviving vertex $(X_1,Y_1,Z_1) \in \mathcal{H}$ with $f(X_1) \vee Y_1 \neq Y_1$. Consider the step $i$ after which the ranks in $\mathcal{H}_i$ of all the vertices $(X_1, f(X_1) \vee Y_1, Z')$ become equal to the rank of $(X_1,Y_1,Z_1)$. See that at this step, localized function $\pi_{XY}^*(\gamma_i(f))^{e_{(X_1,Y_1,Z_1)}}$ on the star of $e_{(X_1,Y_1,Z_1)}$ is equal to minus the characteristic function of any element of $(f(X_1) \vee Y_1) \setminus Y_1$.
\par Then, we verify that the weights are as claimed using induction on applying $\pi_{XY}^* (\gamma_i(f))$. The necessary observations are that
\begin{itemize}
\item all the initial covering relations coming from $M_1$ have weights $1$;
\item all the newly-created edges of the type $((X_1,Y,Z),(X_2,Y,Z))$ between the vertices present in $\mathcal{H}$ also have weights $1$, because $\pi_{XY}^*(\gamma(f))$ coincides with $-rk_{M_1}$ on the segment;
\item all the initial covering relations coming from $\mathcal{G}_{YZ}$ inherit their weights from $\mathcal{G}_{YZ}$, thus, in induction handling equation \ref{eq:graphweight}, everything is multiplied by $\eta_{\mathcal{F}(Y_1,Y_2,Z_1)}(Y_1)$.
\end{itemize}
\end{proof}
According to Lemma \ref{lemma:XYZ} and Definition \ref{def:pushforward}, the weight of the maximal cone generated by the chain $\varnothing \lessdot (X_1,Z_1) \lessdot \ldots \lessdot (E(M_1), E(M_3))$ of length $r = \rk{M_3}$ in the push-forward is given by
\begin{equation} \label{eq:xzweight}
\sum_{Y_0 \leqslant \ldots \leqslant Y_r} \left( \prod_{i=0}^{r-1} \omega_{\mathcal{G}_{YZ}}((Y_i,Y_{i+1}),(Z_i,Z_{i+1})) \cdot
\prod_{i=0}^{r-1} \omega_{\mathcal{G}_{XY}}((X_i,X_{i+1}),(Y_i,Y_{i+1})) \right),
\end{equation}
where we assume summation over only suitable sequences $Y_i$ (mind that, unlike $Z_i$, they do not have to be strictly increasing in $\mathcal{F}(M_2)$). Next, we rearrange
\begin{equation*}
\begin{split}
\prod_{i=0}^{r-1} \omega_{\mathcal{G}_{YZ}}((Y_i,Y_{i+1}),(Z_i,Z_{i+1})) = \prod_{i=0}^{r-1}
\left( \sum_{Y'_i \in \mathcal{F}(Y_i,Y_{i+1},Z_i)} \mu(Y_i,Y'_i) \right) = \\ =
\sum_{Y'_0 \ldots Y'_r} \left( \prod_{i=0}^{r-1} \mu_{\mathcal{F}(Y_i,Y_{i+1},Z_i)}(Y_i,Y'_i) \right),
\end{split}
\end{equation*}
and, likewise,
$$\prod_{i=0}^{r-1} \omega_{\mathcal{G}_{XY}}((X_i,X_{i+1}),(Y_i,Y_{i+1})) =
\sum_{X'_0 \ldots X'_r} \left( \prod_{i=0}^{r-1} \mu_{\mathcal{F}(X_i,X_{i+1},Y_i)}(X_i,X'_i) \right).$$
Substituting into equation \ref{eq:xzweight} and rearranging, we get (dropping the subscripts of $\mu$'s)
\begin{equation} \label{eq:eq:xzweight1}
\sum_{Y_0 \ldots Y_r}\sum_{\substack{Y'_0 \ldots Y'_r \\ X'_0 \ldots X'_r}} \left( \prod_{i=0}^{r-1} \mu(Y_i,Y'_i)\mu(X_i,X'_i) \right).
\end{equation}
The system of sets $X'_i, Y_i, Y'_i$ satisfies the following conditions:
$$\begin{matrix}
X_i \leqslant X'_i \leqslant X_{i+1} & Y_i \leqslant Y_{i+1} & Y_i \leqslant Y'_i \leqslant Y_{i+1} \\
f(X'_i) \leqslant Y_i & g(Y_i) \leqslant Z_i & g(Y_i') \leqslant Z_i
\end{matrix}$$
In the initial sum \ref{eq:xzweight} we've chosen $Y_i$'s satisfying conditions from the middle column and then $X'_i$'s and $Y'_i$'s satisfying the rest of conditions. We obtain the same system of sets by first choosing $X'_i$ such that
\begin{equation} \label{xprimecond}
X_i \leqslant X'_i \leqslant X_{i+1} \; \mbox{ and } \; g(f(X'_i)) \leqslant Z_i,
\end{equation}
and then choosing
\begin{equation} \label{eq:yyprimecond}
\begin{matrix}
 & & f(X'_i) & & & & f(X'_{i+1}) & & & & \\
 & & \vgeq & & & & \vgeq & & & & \\
\ldots & \leqslant & Y_i & \leqslant & Y'_i & \leqslant & Y_{i+1} & \leqslant & Y'_{i+1} & \leqslant & \ldots \\
 & & & & \vmapsto & & & & \vmapsto & & \\
 & & & & g(Y'_i) & \leqslant & Z_i & & g(Y'_{i+1}) & \leqslant & Z'_{i+1}
\end{matrix}
\end{equation}
Thus, the summation in equation \ref{eq:eq:xzweight1} can be rewritten as
\begin{equation*}
\sum_{X'_0 \ldots X'_r} \left( \prod_{i=0}^{r-1} \mu(X_i,X'_i) \cdot \sum_{\substack{Y_0 \ldots Y_r \\ Y'_0 \ldots Y'_r}} \prod_{i=0}^{r-1} \mu(Y_i,Y'_i) \right),
\end{equation*}
while the desired coefficient of the same maximal cone generated by the chain $\varnothing \lessdot (X_1,Z_1) \lessdot \ldots \lessdot (E(M_1), E(M_3))$ in $\Gamma_{g \circ f}$ is just
\begin{equation*}
\prod_{i=0}^{r-1} \omega_{\mathcal{G}_{XZ}}((X_i,X_{i+1}),(Z_i,Z_{i+1})) = \sum_{X'_0 \ldots X'_r} \left( \prod_{i=0}^{r-1} \mu(X_i,X'_i) \right),
\end{equation*}
where the summation is taken over $X'_i$ with same conditions as in \ref{xprimecond}, so we need to show that
\begin{equation*}
\sum_{\substack{Y_0 \ldots Y_r \\ Y'_0 \ldots Y'_r}} \prod_{i=0}^{r-1} \mu(Y_i,Y'_i) = 1.
\end{equation*}
By Lemma \ref{lemma:muproperty}, for a poset $\mathcal{P}$ with bottom element, $\sum_{a \leqslant_{\mathcal{P}} b} \mu(a,b) = 1$. Making choices of pairs $(Y_i,Y'_i)$ one by one, we obtain
\begin{equation*}
\begin{split}
\sum_{\substack{Y_0 \ldots Y_r \\ Y'_0 \ldots Y'_r}} \prod_{i=0}^{r-1} \mu(Y_i,Y'_i) = \sum_{Y_0,Y'_0} \mu(Y_0,Y'_0) \sum_{\substack{Y_1 \ldots Y_r \\ Y'_1 \ldots Y'_r}} \prod_{i=1}^{r-1} \mu(Y_i,Y'_i) = \\ =
\sum_{Y_0,Y'_0} \mu(Y_0,Y'_0) \sum_{Y_1,Y'_1} \mu(Y_1,Y'_1) \sum_{\substack{Y_2 \ldots Y_r \\ Y'_2 \ldots Y'_r}} \prod_{i=2}^{r-1} \mu(Y_i,Y'_i) \ldots
\end{split}
\end{equation*}
where the summation $(Y_j,Y'_j)$ goes over pairs that satisfy all the conditions of \ref{eq:yyprimecond} concerning $X_i$ and $Z_i$, and also $Y_j \geqslant Y'_{j-1}$. Each of those modified subsets has the minimal element $Y'_{j-1} \vee f(X'_j)$, therefore, the whole sum turns into $1$ from the end to the beginning.
\end{proof}
\par
\begin{remark}
It is possible to define $\Gamma_f$ without $f$ itself, using the description from Theorem \ref{th:grcorstr}, thus avoiding its proof. There are several reasons, though, for which we find our approach preferable. Firstly, the fact that $\Gamma_f$ is a Flag tropical fan is now an instance of Lemma \ref{lemma:ff}, and we do not need to verify balancing conditions of $\Gamma_f$ combinatorially. Secondly, using $f$ to construct $\Gamma_f$ turns graph correspondence into a rigorous generalization of diagonal $\Delta \subset B(N) \times B(N)$ and Bergman subfan $B(N) \subset B(M)$ from \cite{FRau} (Example \ref{ex:grcors} (2)). Thirdly, $\Gamma_f$ is defined even if $f$ is not a covering lattice map, while it is not a Flag fan and does not admit this kind of description --- this may turn out to be valuable to us (see Question \ref{question} below). Finally, $f$ clearly stores information about $\Gamma_f$ in a more compact way. The next statement hints that, maybe, we can even work with $\gamma(f)$ instead of $f$ when trying to generalize claims like Theorem \ref{th:functor}.
\end{remark}
\begin{lemma}
If $f,g$ are covering lattice maps, then $\gamma(g \circ f)$ can be defined as
$$\gamma(g \circ f)(F_X,F_Z) = \min_{F_Y \in \mathcal{F}(M_2)}{(\gamma(g)(F_Y,F_Z)+\gamma(f)(F_X,F_Y)+\rk_{M_2}{F_Y})}.$$
\end{lemma}
\begin{proof}
By Definition \ref{def:graphcor} we have
\begin{equation*}
\begin{split}
\gamma(f)((F_X,F_Y)) = \rk_{M_2}(f(F_X) \cup F_Y) - \rk_{M_1} F_X - \rk_{M_2} F_Y, \\
\gamma(g)((F_Y,F_Z)) = \rk_{M_3}(g(F_Y) \cup F_Z) - \rk_{M_2} F_Y - \rk_{M_3} F_Z.
\end{split}
\end{equation*}
Substituting $F_Y = f(F_X)$, we get
\begin{equation*}
\begin{split}
\gamma(g)(f(F_X),F_Z)+\gamma(f)(F_X,f(F_X)) + \rk_{M_2}{f(F_X)}= \\ =
\rk_{M_3}(g(f(F_X)) \cup F_Z) - \rk_{M_1} F_X - \rk_{M_3} F_Z = \gamma(g \circ f)(F_X,F_Z).
\end{split}
\end{equation*}
Thus, it remains to prove that for any other $F' \in \mathcal{F}(M_2)$ the sum is not smaller, which is equivalent to
$$\rk_{M_2}(f(F_X) \cup F') - \rk_{M_2}{F'} \geqslant \rk_{M_3}(g(f(F_X)) \cup F_Z) - \rk_{M_3}(g(F') \cup F_Z).$$
Since $g$ is a covering map, applying it does not increase the difference of ranks, therefore,
$$\rk_{M_2}(f(F_X) \cup F') - \rk_{M_2}{F'} \geqslant \rk_{M_3}(g(f(F_X) \cup F')) - \rk_{M_3}{g(F')} \geqslant \rk_{M_3}(g(f(F_X))) - \rk_{M_3}{g(F')},$$
and after joining both terms with $F_Z$ the difference of ranks cannot increase by Lemma \ref{lemma:rankunion}.
\end{proof}
\begin{remark} \label{remark:notpullback}
As we can see from Theorem \ref{th:functor} and Example \ref{ex:three}, a reasonable subcategory of $\mathbf{SMatrCor}$ will not have tropical fibre product as pullback. More precisely, if $M_1$ is a uniform matroid of rank $3$ on the groundset $\{1,2,3,4\}$, $M_2$ is a uniform matroid of rank $3$ on the groundset $\{1,2,3,5\}$ and $N$ is a uniform matroid of rank $3$ on the groundset $\{1,2,3\}$, then the tropical fibre product is $B(M)$, where $M$ is a uniform matroid of rank $3$ on the groundset $\{1,2,3,4,5\}$. Consider Bergman fan $B(M')$ of another amalgam $M'$ on the groundset $\{1,2,3,4,5\}$ with $4,5$ being parallel elements. Then, there are numerous correspondences $M \vdash M'$ making $4$ and $5$ coincide, exactly as in Example \ref{ex:three}. All of them commute with the projection graph correspondences $M_{1,2} \vdash M,M'$, though.
\end{remark}
\begin{remark}
Note that in the image of the $\mathbf{SMatrWL_{\gtrdot}}$ under the forgetful functor $\mathrm{Pt}$ the tropical fibre product does not satisfy the universal property $\ref{eq:univprop}$. We show this by constructing an example of the proper amalgam and another amalgam such that the unique weak map between their groundset can only come from the non-covering lattice map --- see Figure \ref{fig:notpullback}. One verifies that $M$ is the tropical fibre product, but the identity map of the groundsets $M \to M'$ cannot be obtained from covering lattice map. Indeed, the flat $\{7,8\}$ of $M$ cannot map to anything other than $\{7,8\}$ --- if it maps to, say, $\{1,7,8\}$, then $\{4,5,6,7,8\}$ covers $\{7,8\}$ in $M$, but the closure of $\{1,4,5,6,7,8\}$ in $M'$ is the whole groundset, a flat that does not cover $\{7,8\}$. But then $\{1,2,3,7,8\}$ covers $\{7,8\}$ in $M$ and not in $M'$.
\begin{figure}[h]
\includegraphics[width=\textwidth]{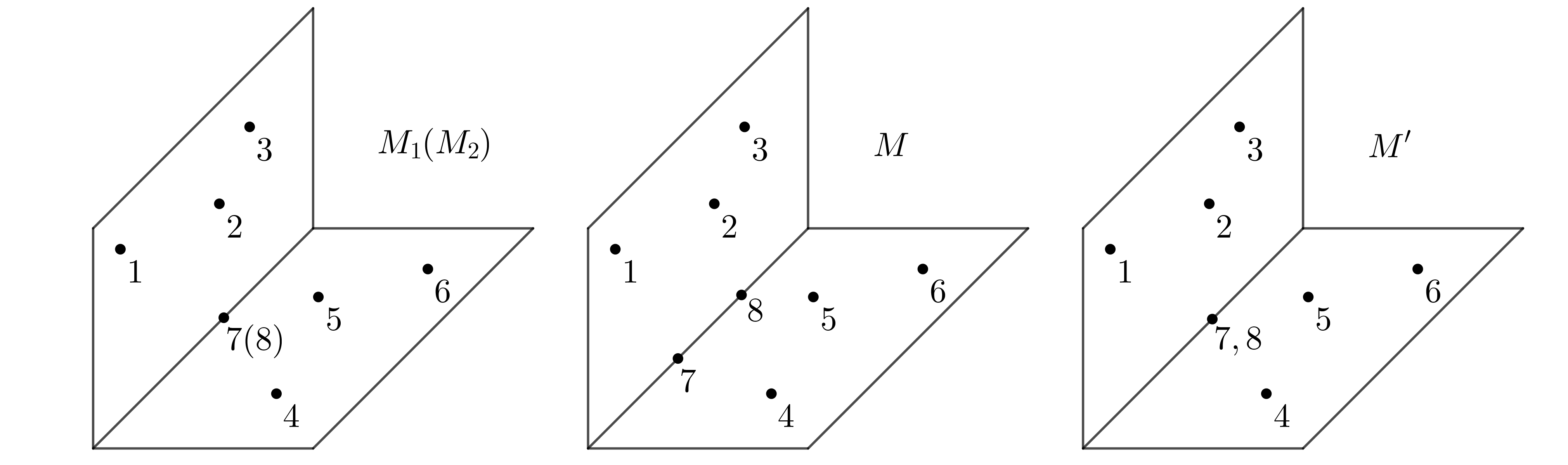}
\caption{$N$ is the restriction of $M_1$ or $M_2$ onto $\{1,2,3,4,5,6\}$.}
\label{fig:notpullback}
\end{figure}
\end{remark}
This insufficiency leads to the following question:
\begin{question} \label{question}
Can it be shown that $f \mapsto \Gamma_f$ is a functor from the whole $\mathbf{SMatrWL}$ to $\mathbf{SMatrCor}$?
\end{question}
\printbibliography
\end{document}